\documentclass{amsart}
\usepackage[margin=2.5cm]{geometry}
\usepackage{amsmath,amsfonts,amssymb,amsthm}
\usepackage{graphics}
\usepackage{epsfig, esint}
\usepackage{graphics,color}
\usepackage[dvipsnames]{xcolor}
\usepackage{paralist}

\usepackage[numbers,sort&compress]{natbib}

   \definecolor{labelkey}{gray}{.8}
   \definecolor{refkey}{gray}{.8}
   \usepackage{hyperref}
\hypersetup{colorlinks, linkcolor=blue, urlcolor=red, filecolor=green, citecolor=blue}

\usepackage{xstring} 
\newcommand{\BR}[1]{%
     \IfEqCase{#1}{%
         {1}{\color{red}}
 {2}{\color{black}}
 {3}{\color{red}}
 {4}{\color{blue}}
 }[\PackageError{BC}{Undefined option to tree: #1}{}]%
 }
\newcommand{\ER}{\color{black}}

\providecommand{\uoset}[3]{\underset{\phantom{#1}}{\overset{#2}{#3}}}
\providecommand{\dt}{\, \mathrm{d} t}

\providecommand{\R}{\mathbb{R}}

\providecommand{\osc}{\operatorname{osc}}
\providecommand{\ee}{{\rm e}}

\newcommand{\e}{\varepsilon}

\newcommand{\step}[1]{\medskip\noindent\textbf{Step #1. }}
\newcommand{\substep}[1]{\medskip\noindent\textit{Substep #1. }}

\providecommand{\PfStart}[1]{\newcounter{#1}}
\providecommand{\PfStep}[2]{\refstepcounter{#1} \ifnum\value{#1}=1
\else\medskip\noindent\fi\textbf{Step \arabic{#1}.}\label{#2}}

\newcommand{\ignore}[1]{}

\newtheorem{proposition}{Proposition}
\newtheorem{theorem}{Theorem}
\newtheorem{remark}{Remark}
\newtheorem{lemma}{Lemma}
\newtheorem{corollary}{Corollary}
\newtheorem{assumption}{Assumption}

\author[P.~Bella]{Peter Bella}
\address{Fakult\"at f\"ur Mathematik, Technische Universit\"at Dortmund}
\email{peter.bella@tu-dortmund.de}
\author[M. Sch\"affner]{Mathias Sch\"affner}
\address{Fakult\"at f\"ur Mathematik, Technische Universit\"at Dortmund}
\email{mathias.schaeffner@tu-dortmund.de}

\title[Non-uniformly parabolic equations and the RCM]{Non-uniformly parabolic equations and applications to the random conductance model} 

\begin{document}
\maketitle


\begin{abstract}
 We study local regularity properties of linear, non-uniformly parabolic finite-difference operators in divergence form related to the random conductance model on $\mathbb Z^d$. In particular, we provide an oscillation decay assuming only certain summability properties of the conductances 
 and their inverse,  thus improving  recent results in that direction. As an application, we provide a local limit theorem for the random walk in a random  degenerate and unbounded \ER environment.
 
 \end{abstract}

\section{Introduction}

In this contribution, we continue our research \cite{BS19a,BS19b} on regularity and stochastic homogenization of non-uniformly elliptic equations. In \cite{BS19a}, we studied local regularity properties of weak solutions of elliptic equations in divergence form
\begin{equation}
\nabla\cdot a\nabla u=0
\end{equation}
and proved local boundedness and the validity of Harnack inequality under essentially minimal integrability conditions on the ellipticity of the coefficients $a$. This generalizes the seminal theory of De Giorgi, Nash and Moser \cite{DG57,Nash58,Moser61} and improves in an optimal way classic results due to Trudinger \cite{T71} (see also \cite{MS68}). In \cite{BS19b}, we adapted the regularity theory from \cite{BS19a} to discrete finite-difference equations in divergence form and used this to obtain a quenched invariance principle for random walks among random degenerate conductances (see the next section for details).

In the present contribution, we extend our previous results in two ways
\begin{enumerate}[(A)]
 \item (deterministic part) We establish local regularity properties in the sense of an oscillation decay (and thus H\"older-continuity) for solution of \textit{discrete} version of parabolic equations 
\begin{equation}
\partial_t u- \nabla\cdot a\nabla u=0
\end{equation}
under relaxed ellipticity conditions compared to very recent contributions in the field (see e.g.\ \cite{CS19,ADS16,AT19,CD15}). 
 \item (random part) Based on the regularity result in (A), we establish a local limit theorem for random walks among degenerate and unbounded random conductances.
\end{enumerate}

\subsection{Setting and main deterministic regularity results}

In this paper we study the nearest-neighbor random conductance model on the $d$-dimensional Euclidean lattice $(\mathbb Z^d,\mathbb B^d)$, for $d\geq2$. Here $\mathbb B^d$ is given by the set of nonoriented nearest-neighbor bounds, that is $\mathbb B^d:=\{\{x,y\}\,|\,x,y\in\mathbb Z^d,\,|x-y|=1\}$. 

We set 
\begin{equation}\label{def:omega}
 \Omega:=(0,\infty)^{\mathbb B^d}
\end{equation}
and call $\omega(\ee)$ the \textit{conductance} of the bond $\ee\in\mathbb B^d$ for every $\omega=\{\omega(\ee)\,|\,\ee\in \mathbb B^d\}\in\Omega$. To lighten the notation, for any $x,y\in\mathbb Z^d$, we set
\begin{equation*}
\omega(x,y)=\omega(y,x):=\omega(\{x,y\})\quad\forall \{x,y\}\in\mathbb B^d,\qquad \omega(\{x,y\})=0\quad\forall \{x,y\}\notin\mathbb B^d.
\end{equation*}
We study local regularity properties of functions $u: Q\subset \R\times\mathbb  Z^d\to\R$ satisfying the parabolic finite-difference equation
\begin{equation*}
 \partial_t u - \mathcal L^\omega u=0, 
\end{equation*}
where $\mathcal L^\omega$ is the elliptic operator defined by
\begin{equation}\label{def:Lomega}
 (\mathcal L^\omega u)(x)=\sum_{y\in\mathbb Z^d}\omega(x,y)(u(y)-u(x)).
\end{equation}
We emphasize here that $\mathcal L^\omega$ is in fact an elliptic finite-difference operator in divergence form, see \eqref{Lomegadiv} below. Our main deterministic regularity result is the following (see Section~\ref{sec:notation} for  notation).
\begin{theorem}[Parabolic oscillation decay]\label{T:oscillation}
Fix $d\geq2$, $\omega\in\Omega$ and $p\in(1,\infty]$, $q\in(\frac{d}2,\infty]$ satisfying $\frac1p+\frac1q<\frac2{d-1}$. Then there exist $N=N(d)\in\mathbb N$ and $\theta_{\rm P}:(0,\infty)\times(0,\infty)\to(0,1)$ which is continuous and monotonically increasing in both variables such that the following is true: Let $u$ be such that
\begin{equation*} 
\partial_t u-\mathcal L^\omega u=0 \quad \textrm{in} \quad Q(n):=[t_0-n^2,t_0]\times B(x_0,n),
\end{equation*}
for some $t_0\in\R$, $x_0\in\mathbb Z^d$ and $n\ge N$. Then
\begin{equation*}
 {\rm osc}\,(u,Q(\tfrac18n))\leq \theta_{\rm P}\, {\rm osc}\,(u,Q(n)),
\end{equation*} 
where $\theta_{\rm P}=\theta_{\rm P}(\|\omega\|_{\underline L^p(B(x_0,n))},\|\omega^{-1}\|_{\underline L^q(B(x_0,n))})$ and ${\rm osc}\,(u,Q):=\max_{(t,x)\in Q} u(t,x)-\min_{(t,x)\in Q} u(t,x)$ denotes the oscillation of $u$.
\end{theorem}
\begin{remark}
The restrictions on the exponents $p$ and $q$ in Theorem~\ref{T:oscillation} are natural in the sense that they are  essentially necessary in order to establish local boundedness for solutions of $\partial_t u- \mathcal L^\omega u=0$, see Remarks~\ref{rem:Tboundpqoptimal} and \ref{rem:trapping} below.

\smallskip

In recent works \cite{AT19,CS19}, the conclusion of Theorem~\ref{T:oscillation} is contained under the more restrictive relation $\frac1p+\frac1q<\frac2d$. Note that \cite{CS19} also contains results for time-depending conductances and \cite{AT19} allows for more general speed measures. It would be interesting to see to which extend the method of the present paper can also yield improvements in these cases. 
\end{remark}
Obviously, Theorem~\ref{T:oscillation} applies also to $\mathcal L^\omega$-harmonic functions. However, it turns out that in the elliptic case a slightly more precise result can be proven under weaker assumptions:
\begin{theorem}[Elliptic oscillation decay]\label{P:oscillationelliptic}
Let $d\ge2$, and if $d \ge 3$ let $p,q\in(1,\infty]$ be such that $\frac1p+\frac1q<\frac2{d-1}$. Then there exists  $\theta_{\rm E}:[1,\infty)\to(0,1)$, which is continuous and monotonically increasing, such that the following holds: Let $\omega \in \Omega$ and $u$ solves $\mathcal L^\omega u=0$ in $B(x_0,4n)$ for some $x_0\in\mathbb Z^d$. Then,
%
\begin{equation*}
 {\rm osc}\,(u,B(x_0,n))\leq \theta_{\rm E}\, {\rm osc}\,(u,B(x_0,4n)),
\end{equation*} 
where
\begin{equation*}
 \theta_{\rm E}=\begin{cases} \theta_{\rm E}(\|\omega\|_{\underline L^1(B(x_0,4n))}\|\omega^{-1}\|_{\underline L^1(B(x_0,4n))})&\mbox{if $d=2$}\\\theta_{\rm E}(\|\omega\|_{\underline L^p(B(x_0,4n))}\|\omega^{-1}\|_{\underline L^q(B(x_0,4n))})&\mbox{if $d\geq3$}
\end{cases}.
\end{equation*}
\end{theorem}
\begin{remark}
In \cite{BS19a}, the corresponding statement of Theorem~\ref{T:oscillation} for $d\geq3$ is proven in the continuum setting as a consequence of elliptic Harnack inequality. Note that, in $d=2$ we can consider the borderline case $p=q=1$ for which we did not proved Harnack inequality in \cite{BS19a}. Previously, elliptic Harnack inequality and thus oscillation decay in the form of Theorem~\ref{P:oscillationelliptic} was proven in \cite{ADS16} under the more restrictive relation $\frac1p+\frac1q<\frac2d$ (see also the classic paper \cite{T71}).
\end{remark}

\subsection{Local limit theorem}

In what follows we consider random conductances $\omega$ that are distributed according to a probability measure $\mathbb P$  on $\Omega$ equipped with the $\sigma$-algebra $\mathcal F:=\mathcal B((0,\infty))^{\otimes \mathbb B^d}$ and we write $\mathbb E$ for the expectation with respect to $\mathbb P$. We introduce the group of space shifts $\{\tau_x:\Omega\to\Omega\,|\,x\in\mathbb Z^d\}$ defined by
\begin{equation}\label{def:shift}
 \tau_x\omega(\cdot):=\omega(\cdot+x)\qquad\mbox{where for any $\ee=\{\underline \ee,\overline\ee\}\in\mathbb B^d$, $\ee+x:=\{\underline e+x,\overline\ee+x\}\in\mathbb B^d$.}
\end{equation}
For any fixed realization $\omega$, we study the reversible continuous time Markov chain, $X = \{X_t : t \geq0\}$, on $\mathbb Z^d$ with generator $\mathcal L^\omega$ given in \eqref{def:Lomega}. Following \cite{ADS15}, we denote by ${\bf P}_x^\omega$ the law of the process starting at the vertex $x\in\mathbb Z^d$ and by ${\bf E}_x^\omega$ the corresponding expectation. $X$ is called the \textit{variable speed random walk} (VSRW) in the literature since it waits at $x\in\mathbb Z^d$ an exponential time with mean $1/\mu^\omega(x)$, where $\mu^\omega(x)=\sum_{y\in\mathbb Z^d}\omega(x,y)$ and chooses its next position $y$ with probability $p^\omega(x,y):=\omega(x,y)/\mu^\omega(x)$.

\begin{assumption}\label{ass}
Assume that $\mathbb P$ satisfies the following conditions:
\begin{enumerate}[(i)]
 \item (stationary) $\mathbb P$ is stationary with respect to shifts, that is $\mathbb P\circ \tau_x^{-1}=\mathbb P$ for all $x\in\mathbb Z^d$.
 \item (ergodicity) $\mathbb P$ is ergodic, that is $\mathbb P[A]\in\{0,1\}$ for any $A\in\mathcal F$ such that $\tau_x(A)=A$ for all $x\in\mathbb Z^d$
 %
\end{enumerate}
\end{assumption}

Starting with the seminal contribution~\cite{SS04}, a considerable effort has been invested in the derivation of quenched invariance principles under various assumptions on the conductances, see the surveys \cite{Biskup,Kumagai} and the discussion below. The following quenched invariance principe is the starting point for the probabilistic aspects of our contribution.

\begin{theorem}[Quenched invariance principle, \cite{BS19b,Biskup}]\label{T}
 Suppose $d\geq2$ and that Assumption~\ref{ass} is satisfied. Moreover, suppose that there exists $p,q\in [1,\infty]$ satisfying $\frac1p+\frac1q<\frac2{d-1}$ such that
 \begin{equation*}
 \mathbb E[\omega(\ee)^p]<\infty,\quad \mathbb E[\omega(\ee)^{-q}]<\infty\qquad\mbox{for any $\ee\in\mathbb B^d$}.
 \end{equation*}
 For $n\in\mathbb N$, set  $X_t^{(n)}:=\frac1n X_{n^2t}$, $t\geq0$. Then, for $\mathbb P$-a.e.\ $\omega$ under ${\bf P}_0^\omega$, $X^{(n)}$ converges in law to a Brownian motion on $\R^d$ with a deterministic and non-degenerate covariance matrix $\Sigma^2$.
\end{theorem}

\begin{proof}
 For $d\geq3$ this is \cite[Theorem~2]{BS19b} and for $d=2$ this can be found in \cite{Biskup}.

\end{proof}


In this contribution, we provide a refined convergence statement under slightly stronger moment conditions - namely a local limit theorem. Consider the heat-kernel $p^\omega$ of $X$, characterized by
\begin{equation}\label{def:pt}
 p_t^\omega(x,y)=p^\omega(t;x,y)={\bf P}_x^\omega[X_t=y]\qquad\mbox{for $t\geq0$ and $x,y\in\mathbb Z^d$}.
\end{equation}
The local limit theorem is essentially a pointwise convergence result of the (suitably scaled) heat kernel $p^\omega$ of $X$ towards the Gaussian transition density of the limiting Brownian motion of Theorem~\ref{T}.  Set
\begin{equation}\label{def:kt}
 k_t(x):=k_t^{\Sigma}(x):=\frac1{\sqrt{(2\pi t)^d\det(\Sigma^2)}}\exp\biggl(-\frac{x\cdot (\Sigma^2)^{-1}x}{2t}\biggr),
\end{equation}
where $\Sigma$ as in Theorem~\ref{T}.  

\begin{assumption}\label{ass2}
There exists $p\in(1,\infty]$ and $q\in(\frac{d}2,\infty]$ satisfying 
 \begin{equation}\label{eq:pq}
 \frac1p+\frac1q<\frac2{d-1}
 \end{equation}
 such that
 \begin{equation}\label{ass:moment}
 \mathbb E[\omega(\ee)^p]<\infty,\quad \mathbb E[\omega(\ee)^{-q}]<\infty\qquad\mbox{for any $\ee\in\mathbb B^d$}.
 \end{equation}
\end{assumption}

Now we are in position to state our main probabilistic result:
\begin{theorem}[Quenched local limit theorem]\label{T:Locallimit}
Suppose that Assumptions~\ref{ass} and \ref{ass2} are satisfied. For given compact sets $I\subset (0,\infty)$ and $K\subset\R^d$ it holds
\begin{equation}\label{eq:claimlocallimit}
 \lim_{n\to\infty}\max_{x\in K}\sup_{t\in I}|n^dp_{n^2t}^\omega(0,\lfloor nx\rfloor)-k_t(x)|=0,\quad\mbox{$\mathbb P$-a.s..}
\end{equation}
\end{theorem}

Assuming the stronger condition $\frac1p+\frac1q<\frac2d$ instead of~\eqref{eq:pq}, the conclusion of Theorem~\ref{T:Locallimit} was recently proved by Andres and Taylor \cite{AT19} (for related results in the continuum setting see~\cite{CD15}). Previously, Barlow and Hambly~\cite{BH09} gave general criteria for a local limit theorem to hold. These criteria were applied to uniformly elliptic conductances or supercritical i.i.d.\ percolation clusters; see \cite{CH09} for further generalizations. In~\cite{BKM15}, Boukhadra, Kumagai, and Mathieu identified sharp conditions on the tails of i.i.d.\ conductances at zero under which the parabolic Harnack inequality and the local limit theorem hold. 
An inspiring result for the present contribution is \cite{ADS16}, where the local limit theorem for the \textit{constant speed random walk} (CSRW) is proven under Assumption~\ref{ass} and Assumption~\ref{ass2} with \eqref{eq:pq} replaced by $\frac1p+\frac1q<\frac2d$, where the latter turns out be optimal in that case. \ER

We conclude this introduction by mentioning other related results: As mentioned above the quenched invariance principle in the form of Theorem~\ref{T}, for uniformly elliptic conductances (that is $p=q=\infty$) or on supercritical i.i.d.\ percolation clusters, was proven by Sidoravicius and Sznitman~\cite{SS04}. In the special case of i.i.d.\ conductances, that is when $\mathbb P$ is the product measure, which includes e.g.\ percolation models, building on the previous works \cite{BD10,BB07,BP07,M08,MP07}, Andres, Barlow, Deuschel, and Hambly~\cite{ABDH13} showed that the quenched invariance principle holds provided that $\mathbb P[\omega(\ee) > 0] > p_c$ with $p_c=p_c(d)$ being the bond percolation threshold.  In particular, due to independence of conductances such situation is very different as they do not require any moment conditions such as \eqref{ass:moment}. In the general ergodic situation, it is known that at least first moments of $\omega$ and $\omega^{-1}$ are necessary for a quenched invariance principle to hold (see \cite{BBT16}). 
Andres, Deuschel, and Slowik~\cite{ADS15} obtained the conclusion of Theorem~\ref{T} under more restrictive relation $\frac1p+\frac1q<\frac2d$, see also the very recent extension beyond the nearest-neighbor conductance models \cite{BK14}. For quenched invariance principles in dynamic environments, see \cite{ACDS18,BR18}, for a recent paper on local limit theorem, see~\cite{CS19}, as well as \cite{MO16} for related results. A quantitative quenched invariance principle (under quantified ergodicity assumptions) with degenerate conductances can be found in \cite{AN17}. Very recently, building on \cite{AD18}, an almost optimal quantitative local limit theorem in the percolation setting was proven by Dario and Gu~\cite{DG20}. Further results in the stationary \& ergodic setting under moment conditions include large-scale regularity \cite{BFO18}, homogenization in the sense of $\Gamma$-convergence \cite{NSS17}, or spectral homogenization \cite{FHS17}.
\ER

\subsection{Notation}\label{sec:notation}

\begin{itemize}
\item (Sets and $L^p$ spaces) For $y\in\mathbb Z^d$, $n\geq0$, we set $B(y,n):=y+([-n,n]\cap\mathbb Z)^d$ with the shorthand $B(n)=B(0,n)$. For any $S\subset \mathbb Z^d$ we denote by $S_{\mathbb B^d}\subset \mathbb B^d$ the set of bonds for which both end-points are contained in $S$, i.e.\ $S_{\mathbb B^d}:=\{\ee=\{\underline \ee,\overline \ee\}\in\mathbb B^d\,|\, \underline \ee,\overline\ee \in S\}$. For any $S\subset\mathbb Z^d$, we set $\partial S:=\{x\in S\;|\; \exists y\in\mathbb Z^d\setminus S\mbox{ s.t. } \{x,y\}\in\mathbb B^d\}$. Given $p\in(0,\infty)$, $S\subset \mathbb Z^d$, we set for any $f:\mathbb Z^d\to\R^d$ and $F:\mathbb B^d\to\R$ 
\begin{equation*}
 \|f\|_{L^p(S)}:=\left(\sum_{x\in S}|f(x)|^p\right)^\frac1p,\quad \|F\|_{L^p(S_{\mathbb B^d})}:=\left(\sum_{\ee\in S_{\mathbb B^d}}|F(\ee)|^p\right)^\frac1p,
\end{equation*}
and $\|f\|_{L^\infty(S)}=\sup_{x\in S}|f(x)|$. Moreover, normalized versions of $\|\cdot||_{L^p}$ are defined for any finite subset $S\subset \mathbb Z^d$ and $p\in(0,\infty)$ by
\begin{equation*}
 \|f\|_{\underline L^p(S)}:=\left(\frac1{|S|}\sum_{x\in S}|f(x)|^p\right)^\frac1p,\quad \|F\|_{\underline L^p(S_{\mathbb B^d})}:=\left(\frac1{|S_{\mathbb B^d}|}\sum_{\ee\in S_{\mathbb B^d}}|F(\ee)|^p\right)^\frac1p,
\end{equation*}
where $|S|$ and $|S_{\mathbb B^d}|$ denote the cardinality of $S$ and $S_{\mathbb B^d}$, respectively. Throughout the paper we drop the subscript in $S_{\mathbb B^d}$ if the context is clear and \BR2 we set $\|\cdot\|_{\underline L^\infty(S)}:=\|\cdot\|_{ L^\infty(S)}$\ER. Moreover, 
\begin{equation}\label{def:m}
 \forall Q=I\times S\subset \R\times \mathbb Z^d\quad\mbox{we set}\quad{\rm m}(Q):=|I||S|,
\end{equation}
where $|I|$ denotes the Lebesgue measure of $I$ and $|S|$ the cardinality of $S$. 
%
%

\item (discrete calculus) For any bond $\ee\in\mathbb B^d$, we denote by $\underline \ee,\overline \ee\in\mathbb Z^d$ the (unique) vertices satisfying $\ee=\{\underline \ee,\overline \ee\}$ and $\overline \ee - \underline \ee\in\{e_1,\dots,e_d\}$. For $f:\mathbb Z^d\to\R$, we define its \textit{discrete derivative} as
\begin{equation*}
 \nabla f:\mathbb B^d\to\R,\qquad \nabla f(\ee):=f(\overline \ee)-f(\underline \ee).
\end{equation*} 
For $f,g:\mathbb Z^d\to\R$ the following discrete product rule is valid
\begin{align}\label{chainrule}
 \nabla (fg)(\ee)=f(\overline \ee)\nabla g(\ee)+g(\underline \ee)\nabla f(\ee)=f(\ee)\nabla g(\ee)+g(\ee)\nabla f(\ee),
\end{align} 
where we use for the last equality the convenient identification of a function $h:\mathbb Z^d\to\R$ with the function $h:\mathbb B^d\to \R$ defined by the corresponding arithmetic mean
\begin{equation*}
 h(\ee):=\frac12 (h(\overline \ee)+h(\underline \ee)).
\end{equation*}
The \textit{discrete divergence} is defined for every $F:\mathbb B^d\to\R$ as
\begin{equation*}
\nabla^*F(x):=\sum_{\ee\in\mathbb B^d\atop \overline \ee=x}F(\ee)-\sum_{\ee\in\mathbb B^d\atop \underline \ee=x}F(\ee)=\sum_{i=1}^d\left(F(\{x-e_i,x\})-F(\{x,x+e_i\})\right).
\end{equation*}
Note that for every $f:\mathbb Z^d\to\R$ that is non-zero only on finitely many vertices and every $F:\mathbb B^d\to\R$ it holds
\begin{equation}\label{sumbyparts}
 \sum_{\ee\in\mathbb B^d}\nabla f(\ee)F(\ee)=\sum_{x\in\mathbb Z^d}f(x)\nabla^*F(x).
\end{equation}
Finally, we observe that the generator $\mathcal L^\omega$ defined in \eqref{def:Lomega} can be written as a second order finite-difference operator in divergence form, in particular
\begin{equation}\label{Lomegadiv}
\forall u:\mathbb Z^d\to\R\qquad \mathcal L^\omega u(x)=-\nabla^*(\omega\nabla u)(x)\quad\mbox{for all $x\in\mathbb Z^d$.}
\end{equation}
\item (Functions) For a function $u:I\times V\to\R$ with $I\subset \R$ and $V\subset\mathbb Z^d$, we denote by $u_t$ the function $u_t:V\to\R$ given by $u_t=u(t,\cdot)$. We call $u:I\times V\to \R$ caloric (subcaloric or supercaloric) in $Q=I\times V$ if 
\begin{equation*}
 (\partial_t-\mathcal L^\omega)u=0\quad(\mbox{$\leq$ or $\geq$}) \quad\mbox{in $Q$}.
\end{equation*}
Moreover, we call $u:\mathbb Z^d\to\R$ harmonic (subharmonic or superharmonic) in $V$ if 
\begin{equation*}
 -\mathcal L^\omega u=0\quad(\mbox{$\leq$ or $\geq$}) \quad\mbox{in $V$}.
\end{equation*}
\end{itemize}

\section{Parabolic regularity}
%

\subsection{Auxiliary Results}

We recall suitable versions of Sobolev inequality, see Proposition~\ref{T:sob}, and provide an optimization result, formulated in Lemma~\ref{L:optimcutoff} below, that is central in our proof of Theorem~\ref{T:oscillation}. 
\begin{proposition}[Sobolev inequalities]\label{T:sob}
 Fix $d\geq2$. For every $s\in[1,d)$ set $s_d^*:=\frac{ds}{d-s}$.
 \begin{enumerate}[(i)]
  \item For every $s\in[1,d)$ there exists $c=c(d,s)\in[1,\infty)$ such that for every $f:\mathbb Z^d\to\R$ it holds
 \begin{equation}\label{est:sobolev:bulk}
 \|f-(f)_{B(n)}\|_{L^{s_d^*}(B(n))}\leq c\|\nabla f\|_{L^s(B(n))},
 \end{equation}
 where $(f)_{B(n)}:=\frac1{|B(n)|}\sum_{x\in B(n)}f(x)$.
 \item For every $s\in[1,d-1)$ there exists $c=c(d,s)\in[1,\infty)$ such that for every $f:\mathbb Z^d\to\R$ it holds
 \begin{equation}\label{est:sobolev:sphere}
 \|f\|_{L^{s_{d-1}^*}(\partial B(n))}\leq c(\|\nabla f\|_{L^s(\partial B(n))}+n^{-1}\| f\|_{L^s(\partial B(n))}).
 \end{equation}
 \end{enumerate}
 \end{proposition} 
 \ER
Estimate~\eqref{est:sobolev:sphere} is the discrete analogue of the classical Sobolev inequality on the sphere, since the first and the third term measure $f$ on the boundary of the ball/cube (i.e. sphere) whereas the middle term measures $\nabla f$ within this set. The statements of Proposition~\ref{T:sob} are standard and the proof can be found e.g.\ in \cite[Theorem~3]{BS19b}.
 
\begin{lemma}\label{L:optimcutoff}
Fix $d\geq1$, $\rho,\sigma \in \mathbb N$ with $\rho<\sigma$ and $v:\mathbb Z^d \to[0,\infty)$. Consider
\begin{equation*}
\begin{split}
 J(\rho,\sigma,v):=\inf\bigg\{&\sum_{\ee \in \mathbb B^d}(\nabla \eta(\ee))^2 v(\ee) \;|\;\eta:\mathbb Z^d\to[0,1],\\
 &\quad \mbox{$\eta=1$ in $B(\rho)$ and $\eta=0$ in $\mathbb Z^d\setminus B(\sigma-1)$}\bigg\}.
 \end{split}
\end{equation*}
For every $\delta>0$ it holds
\begin{equation}\label{est:L:optimcutoff}
 J(\rho,\sigma,v)\leq (\sigma-\rho)^{-(1+\frac1\delta)}\biggl(\sum_{k=\rho}^{\sigma-1}\biggl(\sum_{\ee\in S(k)}v(\ee)\biggr)^\delta\biggr)^\frac1\delta,
 \end{equation}
where for every $m\in\mathbb N$
\begin{equation}\label{def:Sm}
S(m):=\{\ee\in\mathbb B^d\,|\,\underline \ee\in \partial B(m),\, \overline \ee \in \partial B(m+1)\}.
\end{equation}
\end{lemma}

\begin{proof}[Proof of Lemma~\ref{L:optimcutoff}]
Inequality \eqref{est:L:optimcutoff} was already proven in \cite[Step~1 of the proof of Lemma 1]{BS19b}. For convenience for the reader we recall the computations below.

Restricting the class of admissible cut-off functions to those of the form $\varphi(x)=\hat \varphi(\max_{i=1,\dots,d}\{|x\cdot e_i|\})$, we obtain
\begin{align}\label{1dmin1}
 J(\rho,\sigma,v)\leq& \inf\biggl\{\sum_{k=\rho}^{\sigma-1} (\hat \varphi'(k))^2\sum_{\ee\in S(k)} v(\ee) \;|\;\hat\varphi:\mathbb N\to[0,\infty),\,\hat\varphi(\rho)=1,\,\hat\varphi(\sigma)=0\biggr\}=:J_{\rm 1d},
\end{align}
where $\hat\varphi'(k):=\hat\varphi(k+1)-\hat\varphi(k)$. The minimization problem \eqref{1dmin1} can be solved explicitly. Indeed, set $f(k):=\sum_{\ee\in S(k)} v(\ee)$ for every $k\in\mathbb Z$ and suppose $f(k)>0$ for every $k\in\{\rho,\dots,\sigma-1\}$. Then, $\hat\varphi:\mathbb N\to[0,\infty)$ defined by $\hat\varphi(i):=1 \textrm{ for } i<\rho$, $\hat\varphi(i):=0 \textrm{ for } i> \sigma$, and 
$$\hat\varphi(i):=1-\left(\sum_{k=\rho}^{\sigma-1}f(k)^{-1}\right)^{-1}\sum_{k=\rho}^{i-1}f(k)^{-1}$$
for $i \in {\rho,\ldots,\sigma}$, is a valid competitor in the minimization problem for $J_{\rm 1d}$ and we obtain
\begin{equation}\label{est:L:optimcutoffdelta1}
 J(\rho,\sigma,v)\leq \biggl(\sum_{k=\rho}^{\sigma-1}\biggl(\sum_{\ee\in S(k)} v(\ee)\biggr)^{-1}\biggr)^{-1}.
\end{equation}
Inequality \eqref{est:L:optimcutoffdelta1} combined with an application of H\"older inequality yield the claimed inequality \eqref{est:L:optimcutoff}. Finally, in the case that $f(k)=\sum_{\ee\in S(k)} v(\ee)=0$ for some $k\in\{\rho,\dots,\sigma-1\}$, we easily obtain $J_{\rm 1d}=0$ and \eqref{est:L:optimcutoff} is trivially satisfied.
\end{proof}

\subsection{Local boundedness}

In this subsection, we establish local boundedness for non-negative subcaloric functions $u$. The results of this section, in particular Lemma~\ref{L:key} below, contain the main technical improvements compared to previous related results, e.g.\ \cite{CS19,AT19,CD15}. In principle, we follow the classical strategy of Moser to obtain the local boundedness. We recall that this strategy is based on (i) Caccioppoli inequalities for (powers of) $u$ (see \eqref{est:basic:sub} below), (ii) application of the Sobolev inequality, and (iii) an iteration argument. As in our previous works \cite{BS19a,BS19b} the improvement is mainly obtained by using certain optimized cut-off functions in the Caccioppoli inequality that allow (appealing to Lemma~\ref{L:optimcutoff}) to use Sobolev inequality on ``spheres'' instead of ``balls''. Unfortunately, the implementation of this strategy is technically much more involved in the parabolic case compared to the elliptic case treated in \cite{BS19a,BS19b}.\ER

 Throughout this section, we use the shorthand
\begin{equation}\label{def:h1}
 \|v\|_{\underline H^{1}(I\times B)}^2:=\|v\|_{\underline H^{1}(I\times B;\omega)}^2:=|B|^{-\frac2d}\|\omega^{-1}\|_{\underline L^q(B)}^{-1}\|v\|_{\underline L^2(I\times B)}^2+\|\sqrt\omega \nabla v||_{\underline L^2(I\times B)}^2. 
\end{equation}
%

\begin{lemma}\label{L:key}
Fix $d\geq2$, $\omega\in\Omega$, $p\in(1,\infty)$ and $q\in(\frac{d}2,\infty)$ satisfying $\frac1p+\frac1q<\frac2{d-1}$. Let $\nu\in(0,1)$, $\gamma>2$ and $\theta>1$ be given by
%
%
\begin{equation}\label{def:nu}
 \nu:=1-\delta_2\biggl(1-\frac1\theta\biggr)\in(0,1),\quad \gamma:=2+\frac1p+\frac1{\theta q}\quad\mbox{with}\quad \theta:=\begin{cases}p&\mbox{if $d=2$,}\\
 1+ p\delta_1&\mbox{if $d\geq3$}
  \end{cases},
\end{equation}
where
\begin{equation}\label{def:delta}
 \delta_1:=\frac2{d-1}-\frac1p-\frac1q>0,\quad \delta_2:=\frac2d-\frac1q>0.
\end{equation}
Then there exist $c=c(d,p,q)\in[1,\infty)$ such that the following is true: Let $n,m\in\mathbb N$ with $n<m\leq 2n$ and $0<s_1<s_2$ be given and consider $I_1:=[-s_1,0]$ and $I_2:=[-s_2,0]$. Let $u\geq0$ be a subcaloric function in $I_2\times B(m)$. Then 
for every $\alpha\geq1$
\begin{align}
&\begin{aligned}
 &\sup_{t\in I_1}s_1^{-1}\|u_t^\alpha\|_{\underline L^{2}(B(n))}^2+\|u^\alpha\|_{\underline H^{1}(I_1\times B(n))}^2\\
  &\ \leq  c \alpha^2 \biggl(\sup_{t\in I_2}s_2^{-1}\|u_t^\alpha\|_{\underline L^2(B(m))}^2\biggr)^{1-\nu} s_2^{1-\nu}\|\omega^{-1}\|_{\underline L^q(B(m))}\biggl(\frac{\|\omega\|_{\underline L^p(B(m))}}{(1-\tfrac{n}m)^{\gamma}}+\frac{m^2}{s_2-s_1}\biggr)\frac{s_2}{s_1}\|u^{\alpha\nu}\|_{\underline H^{1}(I_2\times B(m))}^2;\label{est:keysubsol}
\end{aligned}  
\\[10pt]
&\begin{aligned}
  &\sup_{t\in I_1}s_1^{-1}\|u_t^\alpha\|_{\underline L^{2}(B(n))}^2+\|\sqrt{\omega}\nabla (u^\alpha)\|_{\underline L^2(I_1\times B(n))}^2\\
&\ \leq \frac{c\alpha^2}{m^2}\biggl(\frac{\|\omega\|_{\underline L^p(B(m))}}{(1-\frac{n}{m})^2}+\frac{m^2}{s_2-s_1}\biggr) \frac{s_2}{s_1}\fint_{I_2}\|u_t^\alpha\|_{\underline L^{\frac{2p}{p-1}}( B(m))}^2\dt;\label{est:keysubsol2}
\end{aligned}
\\[10pt]
&\sup_{t\in I_1}s_1^{-1}\|u_t^\alpha\|_{\underline L^{2}(B(n))}^2\leq c\alpha^2 \|\omega^{-1}\|_{\underline L^q(B(m))}\biggl(\frac{\|\omega\|_{\underline L^p(B(m))}}{(1-\tfrac{n}m)^{2+\frac1p+\frac1q}}+\frac{m^2}{s_2-s_1}\biggr) \frac{s_2}{s_1} \|u^{\alpha}\|_{\underline H^{1}(I_2\times B(m))}^2.\label{est:keysubsol3}
\end{align}
\end{lemma}

The main achievement of Lemma~\ref{L:key} is estimate~\eqref{est:keysubsol}, where at the expense of increasing the domain of integration we control $u^\alpha$ in terms of $u^{\nu\alpha}$ with $\nu < 1$. The factor on the right-hand side involving norm of $u^\alpha$ (and not $u^{\nu\alpha}$) to a small power will be dealt with later. The other two estimates~\eqref{est:keysubsol2} and~\eqref{est:keysubsol3} do not include improvement of integrability, and their proofs are significantly simpler. 


\begin{proof}[Proof of Lemma~\ref{L:key}] \PfStart{lm2}
Throughout the proof we write $\lesssim$ if $\leq$ holds up to a positive constant that depends only on $d$, $p$, and $q$. We introduce,
\begin{align*}
 \mathcal A(n,m):=&\{\eta:\mathbb Z^d\to[0,1]\,|\, \eta=1\,\mbox{in $B(n)$ and}\, \eta=0\,\mbox{in $\mathbb Z^d\setminus B(m-1)$}\}.
\end{align*}

\PfStep{lm2}{} We claim that there exists $c=c(d)\in[1,\infty)$ such that for all $\eta\in \mathcal A(n,m)$ and $\alpha\geq1$,
\begin{equation}
\begin{aligned}\label{est:basic:sub}
&\sup_{t\in I_1}\|u_t^\alpha\|_{\underline L^2(B(n))}^2+\int_{I_1}\|\sqrt\omega\nabla (u_t^\alpha)\|_{\underline L^2(B(n))}^2\dt\\
&\leq \frac{c\alpha^2}{|B(m)|}\int_{I_2}\sum_{\ee \in\mathbb B^d}\omega(\ee)(\nabla \eta(\ee))^2(u_t^{\alpha}(\ee))^2\dt+\frac{c}{s_2-s_1}\int_{I_2}\|u_t^\alpha\|_{\underline L^2(B(m))}^2\dt,
\end{aligned}
\end{equation}
where we recall the notation $f(\ee)=\frac12(f(\bar \ee)+f(\underbar \ee))$.  This is a discrete parabolic version of classical Caccioppoli inequality, and is obtained by simply testing the equation with $\eta^2 u$ with $\eta$ being a cutoff-function in space, combined with Cauchy-Schwarz inequality and integration in time.

Since $u\geq0$ is subcaloric, we obtain by the chain rule and estimate \eqref{est:d:chain:2}
\begin{align*}
 \frac1{2\alpha}\frac{\mathrm d}{\mathrm dt}\sum_{x\in\mathbb Z^d}\eta^2(x)u_t^{2\alpha}(x)=&\sum_{x\in\mathbb Z^d}\eta^2(x)u_t^{2\alpha-1}(x)\frac{\mathrm d}{\mathrm dt}u_t(x)\\
 \leq& -\sum_{\ee \in\mathbb B^d}\nabla (\eta^2u_t^{2\alpha-1})(\ee)\omega(\ee)\nabla u_t(\ee)\\
 \leq&-\sum_{\ee \in\mathbb B^d}2\eta(\ee)\nabla \eta(\ee)\omega(\ee)u_t^{2\alpha-1}(\ee)\nabla u_t(\ee)-\frac{2\alpha-1}{\alpha^2}\sum_{\ee \in\mathbb B^d}\eta^2(\ee)\omega(\ee)(\nabla u_t^\alpha(\ee))^2
\end{align*}
and thus
%
\begin{align*}
 & \frac12\frac{\mathrm d }{\mathrm dt}\sum_{x\in\mathbb Z^d}\eta^2(x)u_t^{2\alpha}(x)+\frac{2\alpha-1}{\alpha}\sum_{\ee \in\mathbb B^d}\eta^2(\ee)\omega(\ee)(\nabla u_t^\alpha(\ee))^2\notag\\
 &\uoset{\eqref{est:d:chain:3}}{}{\leq} 2\alpha\sum_{\ee \in\mathbb B^d}\eta(\ee)|\nabla \eta(\ee)|\omega(\ee)u_t^{2\alpha-1}(\ee)|\nabla u_t(\ee)|\notag\\
 & \stackrel{\eqref{est:d:chain:3}}{\leq}  2 \alpha\sum_{\ee \in\mathbb B^d}\eta(\ee)|\nabla\eta(\ee)| \omega(\ee)  u_t^\alpha(\ee)|\nabla u_t^\alpha(\ee)|\notag\\
& \uoset{\eqref{est:d:chain:3}}{}{\leq} \frac{2\alpha-1}{2\alpha}\sum_{\ee \in\mathbb B^d}\eta^2(\ee) \omega(\ee) (\nabla u_t^\alpha(\ee))^2+\frac{2\alpha^3}{2\alpha-1}\sum_{\ee \in \mathbb B^d}\omega(\ee)(u_t^\alpha(\ee))^2(\nabla \eta(\ee))^2,
\end{align*}

where we use in the last estimate Youngs inequality in the form $ab\leq \frac12(\e a^2+\frac1\e b^2)$ with $\e=\frac{2\alpha-1}{2\alpha^2}$. Combining the previous two displays, we obtain
\begin{align}\label{eq:basic:subp1}
 &\frac{\mathrm d}{\mathrm dt}\sum_{x\in\mathbb Z^d}\eta^2(x)u_t^{2\alpha}(x)+\sum_{\ee \in\mathbb B^d}\eta^2(\ee)\omega(\ee)(\nabla u_t^\alpha(\ee))^2 \leq 4 \alpha^2\sum_{\ee\in\mathbb B^d}\omega(\ee)(u_t^\alpha(\ee))^2(\nabla \eta(\ee))^2.
\end{align}
Multiplying \eqref{eq:basic:subp1} with the piecewise smooth function $\zeta$ given by
\begin{equation*}
 \zeta(t):=\begin{cases}0&\mbox{if $t\leq -s_2$}\\ \frac{t+s_2}{s_2-s_1}&\mbox{if $t\in(-s_2,-s_1)$}\\ 1&\mbox{if $t\geq -s_1$}\end{cases}
\end{equation*}
and integrating in time, we obtain \eqref{est:basic:sub} (using $n\leq m\leq2n$ and thus $|B(n)|\lesssim |B(m)|\lesssim |B(n)|$).

\PfStep{lm2}{} Proof of estimate \eqref{est:keysubsol2}. This follows directly from \eqref{est:basic:sub} with $\eta\in \mathcal A(n,m)$ satisfying $|\nabla \eta|\leq 2(m-n)^{-1}$ and H\"older inequality. Indeed, for such a choice of $\eta$ one gets
\begin{align*}
&\sup_{t\in I_1}\|u_t^\alpha\|_{\underline L^2(B(n))}^2+\int_{I_1}\|\sqrt\omega\nabla (u_t^\alpha)\|_{\underline L^2(B(n))}^2\dt\\
&\uoset{\textrm{Jensen}}{\eqref{est:basic:sub}}{\leq} \frac{2c\alpha^2}{|B(m)|(m-n)^2}\int_{I_2}\sum_{\ee \in B(m)}\omega(\ee)(u_t^{\alpha}(\ee))^2\dt+\frac{c}{s_2-s_1}\int_{I_2}\|u_t^\alpha\|_{\underline L^2(B(m))}^2\dt
\\
&\uoset{}{\textrm{H\"older}}{\lesssim} \frac{c\alpha^2\|\omega\|_{\underline L^p(B(m))}}{(m-n)^2} \int_{I_2} \|u_t^\alpha\|_{\underline L^{\frac{2p}{p-1}}(B(m))}^2 \dt+\frac{c}{s_2-s_1}\int_{I_2}\|u_t^\alpha\|_{\underline L^2(B(m))}^2\dt\\
&\uoset{}{\textrm{Jensen}}{\le} \biggl( \frac{c\alpha^2\|\omega\|_{\underline L^p(B(m))}}{(m-n)^2} + \frac{c}{s_2 - s_1}\biggr) \int_{I_2} \|u_t^\alpha\|_{\underline L^{\frac{2p}{p-1}}(B(m))}^2 \dt.
\end{align*}
Using $\alpha \ge 1$ and taking time averages yields~\eqref{est:keysubsol2}.

\PfStep{lm2}{lm2:step2} We claim that there exists $c=c(d,p,q)\in[1,\infty)$ such that
\begin{equation}\label{L1:estkey}
\begin{aligned}
&\min_{\eta\in\mathcal A(n,m)}|B(m)|^{-1}\int_{I_2}\sum_{\ee\in\mathbb B^d} \omega(\ee) (\nabla\eta(\ee))^2(u_t^\alpha(\ee))^2\dt
\\
&\quad \le c\biggl(\sup_{t\in I_2}\|u^\alpha\|_{\underline L^2(B(m))}^2\biggr)^{1-\nu}\|\omega\|_{\underline L^p(B(m))}\|\omega^{-1}\|_{\underline L^q(B(m))}\frac{s_2}{(1-\tfrac{n}m)^{\gamma}}\|u^{\alpha\nu}\|_{\underline H^{1}(I_2\times B(m))}^2.
\end{aligned}
\end{equation}
As in the elliptic case, see \cite[proof of Theorem~4]{BS19b}, the idea is to optimize the cutoff $\eta$ in \eqref{est:basic:sub} via Lemma~\ref{L:optimcutoff} to get spherical averages of $u$ on the right-hand side. Since we do not have good control of time derivatives, the cutoff should be time-independent, hence providing improved integrability for the averages over spheres and \emph{in time}. To ``move'' the time-integral outside we first sacrifice bit of space and time integrability (see Substep 3.1), but which is then dealt with using $L^2$ control of $u$, uniform in time (see Substep 3.2) - which then gives rise to the first term on the right-hand side of~\eqref{L1:estkey}. 
\ER

\substep{\ref{lm2:step2}.1}
%
 We claim that there exists $c_1=c_1(d,p,q)\in[1,\infty)$ such that
\begin{equation}
\begin{aligned}\label{est:L:key:s2:claima}
&\min_{\eta\in\mathcal A(n,m)}\int_{I_2}\sum_{\ee\in\mathbb B^d}\omega(\ee)(\nabla \eta(\ee))^{2}(u_t^{\alpha}(\ee))^{2}\dt \\
&\quad \le c_1\frac{\|\omega\|_{\underline L^p(B(m))}(\|\omega^{-1}\|_{\underline L^q(B(m))}|B(m)|s_2\|u^{\alpha\nu}\|_{\underline H^{1}(I_2\times B(m))}^2)^\frac1\theta}{\left(1-\frac{n}m\right)^{2+\frac1p+\frac1{\theta q}}m^{2(1-\frac1\theta)}} \biggl(\int_{I_2}\|u_t^{\alpha}\|_{L^{2(1+\e)}(B(m))}^{2(1+\e)}\dt\biggr)^{1-\frac1\theta},
\end{aligned}
\end{equation}
where $\theta$, $\nu$ are given in \eqref{def:nu} and
\begin{equation}\label{def:eps}
 \e:=\frac{\delta_2}\theta=\nu-1+\delta_2  > 0.
\end{equation}
Along the proof we also obtain a simpler version (with $\theta=1$) of this inequality:
\begin{equation}
\begin{aligned}\label{est:L:key:s2:claima:simpler}
&\min_{\eta\in\mathcal A(n,m)}\int_{I_2}\sum_{\ee\in\mathbb B^d}\omega(\ee)(\nabla \eta(\ee))^{2}(u_t^{\alpha}(\ee))^{2}\dt \\
&\quad \le c(d,p,q) \frac{\|\omega\|_{\underline L^p(B(m))} \|\omega^{-1}\|_{\underline L^q(B(m))}|B(m)|s_2}{\left(1-\frac{n}m\right)^{2+\frac1p+\frac1q}}\|u^{\alpha}\|_{\underline H^{1}(I_2\times B(m))}^2.
\end{aligned}
\end{equation}

Using Lemma~\ref{L:optimcutoff} with $v(\ee)=\omega(\ee)\int_{I_2}(u_t^\alpha(\ee))^{2}\dt$ we have for every $\delta\in(0,1]$
\begin{align}\label{ineq:caccioptim}
&\min_{\eta\in\mathcal A(m,n)}\int_{I_2}\sum_{\ee\in\mathbb B^d}\omega(\ee)|\nabla \eta(\ee)|^{2}(u_t^\alpha(\ee))^{2}\dt\leq (m-n)^{-(1+\frac1\delta)}\biggl(\sum_{k=n}^{m-1} \biggl(\int_{I_2}\sum_{\ee\in S(k)}\omega(\ee) (u_t^\alpha(\ee))^{2}\dt\biggr)^\delta\biggr)^\frac1\delta,
\end{align}
where $S(k)$ for every $k\in\mathbb N$ is defined in \eqref{def:Sm}. H\"older inequality yields for every $t\in I_2$ and $k\in\{n,\dots,m-1\}$
\begin{align*}
\sum_{\ee \in S(k)} \omega(\ee)(u_t^\alpha(\ee))^{2}\leq& \biggl(\sum_{\ee \in S(k)}\omega(\ee)^p\biggr)^\frac{1}{p}\biggl(\sum_{\ee\in S(k)}(u_t^\alpha(\ee))^{\frac{2p}{p-1}}\biggr)^\frac{p-1}p\\
\lesssim& \biggl(\sum_{\ee \in S(k)}\omega(\ee)^p\biggr)^\frac{1}{p}\biggl(\|u_t^{\alpha}\|_{L^\frac{2p}{p-1}(\partial B(k))}^2+\|u_t^{\alpha}\|_{L^\frac{2p}{p-1}(\partial B(k+1))}^2\biggr).
\end{align*}
Note that the choices for $\theta$ and $\nu$ (see \eqref{def:nu}) yield
\begin{equation}\label{ineq:interpolationrange}
 0<\frac{\theta-\nu}{\theta-1}<\frac{p}{p-1}<\frac{p\nu}{p-\theta}\leq\infty
\end{equation}
(with the understanding $\frac1{0}=\infty$ in the case $d=2$). The inequalities \eqref{ineq:interpolationrange} follow by elementary computations which we provide for the readers convenience in  Substep~3.4 below.

\smallskip

Appealing to \eqref{ineq:interpolationrange} we have the following interpolation inequality 
\begin{equation}\nonumber
 \|u_t^{\alpha}\|_{L^\frac{2p}{p-1}(\partial B(k))}^2\leq \|u_t^{\alpha}\|_{L^\frac{2p\nu}{p-\theta}(\partial B(k))}^\frac{2\nu}\theta\|u_t^{\alpha}\|_{L^\frac{2(\theta-\nu )}{\theta-1}(\partial B(k))}^{2\frac{\theta-\nu}\theta}
\end{equation}
(note that $\frac{p-\theta}{p\nu}\frac{\nu}{\theta}+\frac{\theta-1}{\theta-\nu}\frac{\theta-\nu}{\theta}=1-\frac1p$) and thus by H\"older inequality in time (with exponents $\theta,\frac{\theta}{\theta-1}$), we obtain
\begin{equation}
\begin{aligned}\label{est:lemmakeysobolevsphere}
&\int_{I_2}\sum_{\ee\in S(k)}\omega(\ee) (u_t^\alpha(\ee))^{2}\dt\\
&\lesssim \biggl(\sum_{\ee \in S(k)}\omega(\ee)^p\biggr)^\frac{1}{p}\biggl(\sum_{i=k}^{k+1}\int_{I_2}\|u_t^{\alpha\nu}\|_{L^\frac{2p}{p-\theta}(\partial B(i))}^2\dt\biggr)^\frac1\theta\biggl(\sum_{i=k}^{k+1}\int_{I_2}\|u_t^{\alpha}\|_{L^{2(1+\e)}(\partial B(i))}^{2(1+\e)}\dt\biggr)^{1-\frac1\theta},
\end{aligned}
\end{equation}
where we use
 \begin{equation*}
  \frac{\theta-\nu}{\theta-1}=1+\frac{1-\nu}{\theta-1}\stackrel{\eqref{def:nu}}{=} 1 + \frac{\delta_2}{\theta} \stackrel{\eqref{def:eps}}{=} 1+\e.
 \end{equation*}
We estimate the second factor on the right-hand side in \eqref{est:lemmakeysobolevsphere} by Sobolev inequality: Let $p_*\in[1,2)$ be defined by
$$
\frac1{p_*}=\frac1{d-1}+ \frac{p-\theta}{2p}.
$$
Then a combination of Sobolev and H\"older inequality yield
\begin{equation}
\begin{aligned}\label{est:lemmakeysobolevsphere1}
\|u_t^{\alpha\nu}\|_{L^\frac{2p}{p-\theta}(\partial B(k))}^{2}
&\stackrel{\eqref{est:sobolev:sphere}}{\lesssim} (\|\nabla (u_t^{\alpha\nu})\|_{L^{p_*}(\partial B(k))}^2+k^{-2}\|u_t^{\alpha\nu}\|_{L^{p_*}(\partial B(k))}^2)\\
&\uoset{\eqref{est:sobolev:sphere}}{}{\lesssim} k^{2-(d-1)\frac{\theta}p}\|\omega^{-1}\|_{\underline L^{\frac{p_*}{2-p_*}}(\partial B(k))}\|\sqrt\omega \nabla (u_t^{\alpha\nu})\|_{L^2(\partial B(k))}^2+k^{-(d-1)\frac{\theta}p}\|u_t^{\alpha\nu}\|_{L^2(\partial B(k))}^2,
\end{aligned}
\end{equation}
where we use $k^{d-1}\lesssim |\partial B(k)|\lesssim k^{d-1}$. Combining \eqref{ineq:caccioptim}, \eqref{est:lemmakeysobolevsphere} and \eqref{est:lemmakeysobolevsphere1} with the observation  $\frac{p_*}{2-p_*}\leq q$ (with equality if $d\geq3$), the choice $\delta=(\frac1p+\frac1{\theta q}+1)^{-1}$ and H\"older inequality with exponents $(\frac{p}\delta,\frac{\theta q}{\delta(q+1)},\frac{\theta}{\delta (\theta-1)})$, we obtain
\begin{align}\label{est:lemmakeystep2almostfinal}
(m-n)^{2+\frac1p+\frac1{\theta q}}&\min_{\eta\in\mathcal A(m,n)}\int_{I_2}\sum_{\ee\in\mathbb B^d}\omega(\ee)(\nabla \eta(\ee))^{2}(u_t^\alpha(\ee))^{2}\dt\notag
\\
\uoset{}{\eqref{ineq:caccioptim},\eqref{est:lemmakeysobolevsphere},\eqref{est:lemmakeysobolevsphere1}}{\lesssim}& 
\biggl(\sum_{k=n}^{m-1}  \|\omega\|_{\underline L^p(S(k))}^{\delta}\biggl(\int_{I_2} \sum_{i=k}^{k+1}i^2\|\omega^{-1}\|_{\underline L^q(\partial B(i))}\|\sqrt\omega \nabla (u_t^{\alpha\nu})\|_{L^2(\partial B(i))}^2+\|u_t^{\alpha\nu}\|_{L^2(\partial B(i))}^2\dt\biggr)^\frac\delta\theta\notag
\\
&\quad\times\biggl(\int_{I_2} \sum_{i=k}^{k+1}\|u_t^{\alpha}\|_{L^{2(1+\e)}(\partial B(i))}^{2(1+\e)}\dt\biggr)^{\delta(1-\frac1\theta)}\biggr)^\frac1\delta\notag\\
\uoset{\eqref{ineq:caccioptim},\eqref{est:lemmakeysobolevsphere},\eqref{est:lemmakeysobolevsphere1}}{\textrm{H\"older}}{\lesssim}
&
\biggl(\sum_{k=n}^{m-1}\|\omega\|_{\underline L^p(S(k))}^{p}\biggr)^\frac1p\biggl(\sum_{k=n}^{m-1}\int_{I_2} \sum_{i=k}^{k+1}\|u_t^{\alpha}\|_{L^{2(1+\e)}(\partial B(i))}^{2(1+\e)}\dt\biggr)^{1-\frac1\theta}\notag\\
&\quad\times \biggl(\sum_{k=n}^{m-1}\biggl(\int_{I_2} \sum_{i=k}^{k+1}i^2\|\omega^{-1}\|_{\underline L^q(\partial B(i))}\|\sqrt\omega \nabla (u_t^{\alpha\nu})\|_{L^2(\partial B(i))}^2+\|u_t^{\alpha\nu}\|_{L^2(\partial B(i))}^2\dt\biggr)^\frac{q}{q+1}\biggr)^{\frac{q+1}{q\theta}},
\end{align}
where in the first inequality the factor $k^{-(d-1)\frac \theta p}$ from~\eqref{est:lemmakeysobolevsphere1} is gone due to averaging in $\|\omega\|_{\underline L^p(S(k))}$. To estimate the last factor on the right-hand side in \eqref{est:lemmakeystep2almostfinal}, we first split the sum and then use once more H\"older inequality with exponents $(\frac{q+1}q,q+1)$ to obtain
\begin{equation}
\begin{aligned}\label{est:lemmakeystep2almostfinal1}
 &\biggl(\sum_{k=n}^{m-1}\biggl(\int_{I_2} \sum_{i=k}^{k+1}i^2\|\omega^{-1}\|_{\underline L^q(\partial B(i))}\|\sqrt\omega \nabla (u_t^{\alpha\nu})\|_{L^2(\partial B(i))}^2+\|u_t^{\alpha\nu}\|_{L^2(\partial B(i))}^2\dt\biggr)^\frac{q}{q+1}\biggr)^{\frac{q+1}{q}}\\
 &\lesssim \biggl(\sum_{k=n}^{m}\|\omega^{-1}\|_{\underline L^q(\partial B(k))}^q\biggr)^\frac1{ q}m^2\int_{I_2}\|\sqrt{\omega} \nabla u_t^{\alpha\nu}\|_{L^2(B(m))}^2\dt+(m-n)^\frac1{ q}
 \int_{I_2}\| u_t^{\alpha\nu}\|_{L^2(B(m))}^2\dt.
\end{aligned}
\end{equation}
Combining \eqref{est:lemmakeystep2almostfinal}, \eqref{est:lemmakeystep2almostfinal1}, assumption $m\leq2n$ (and thus $\sum_{k=n}^m\|\cdot\|_{\underline L^s(\partial B(k)}^s\lesssim n^{-(d-1)}\|\cdot\|_{L^s(B(m))}^s\lesssim m\|\cdot\|_{\underline L^s(B(m))}^s$ for all $s\geq1$) and definition \eqref{def:h1}, \ER we obtain
\begin{align*}
&(m-n)^{2+\frac1p+\frac1{\theta q}}\min_{\eta\in\mathcal A(m,n)}\int_{I_2}\sum_{\ee\in\mathbb B^d}\omega(\ee)(\nabla \eta(\ee))^{2}(u_t^\alpha(\ee))^{2}\dt\\
& \lesssim m^{\frac{1}p+\frac{1}{\theta q}}\|\omega\|_{\underline L^p(B(m))}\biggl(\int_{I_2}\|u_t^{\alpha}\|_{L^{2(1+\e)}(B(m))}^{2(1+\e)}\dt\biggr)^{1-\frac1\theta}\\
&\qquad \times \biggl(m^2\|\omega^{-1}\|_{\underline L^q(B(m))}\int_{I_2}\|\sqrt\omega \nabla u_t^{\alpha\nu}\|_{L^2(B(m))}^2\dt+\int_{I_2}\|u_t^{\alpha\nu}\|_{L^2(B(m))}^2\dt\biggr)^{\frac{1}{\theta}}\\
& \lesssim m^{\frac{1}p+\frac{1}{\theta q}}|B(m)|\|\omega\|_{\underline L^p(B(m))}\biggl(\int_{I_2}\|u_t^{\alpha}\|_{\underline L^{2(1+\e)}(B(m))}^{2(1+\e)}\dt\biggr)^{1-\frac1\theta}
\left( m^2\|\omega^{-1}\|_{\underline L^q(B(m))} s_2 \|u_t^{\alpha\nu}\|_{\underline H^1(I_2 \times B(m))}^2\right)^{\frac{1}{\theta}},
\end{align*}
and thus \eqref{est:L:key:s2:claima} follows.

\medskip
The argument for \eqref{est:L:key:s2:claima:simpler} (i.e. the special case $\theta=\nu=1$ of \eqref{est:L:key:s2:claima}) is naturally simpler, since one avoids the $L^{2(1+\e)}$-term.
By Lemma~\ref{L:optimcutoff}, we have for every $\delta\in(0,1]$
\begin{align*}
&\min_{\eta\in\mathcal A(n,m)}\int_{I_2}\sum_{\ee\in\mathbb B^d}\omega(\ee)|\nabla \eta(\ee)|^{2}(u_t^\alpha(\ee))^{2}\dt\leq (m-n)^{-(1+\frac1\delta)}\biggl(\sum_{k=n}^{m-1} \biggl(\int_{I_2}\sum_{\ee\in S(k)}\omega(\ee) (u_t^\alpha(\ee))^{2}\dt\biggr)^\delta\biggr)^\frac1\delta,
\end{align*}
where $S(m)$ is defined in \eqref{def:Sm}. By H\"older inequality, we have
\begin{align*}
\sum_{\ee \in S(k)} \omega(\ee)(u_t^\alpha(\ee))^{2}
\lesssim& \biggl(\sum_{\ee \in S(k)}\omega(\ee)^p\biggr)^\frac{1}{p}\biggl(\|u_t^{\alpha}\|_{L^\frac{2p}{p-1}(\partial B(k))}^2+\|u_t^{\alpha}\|_{L^\frac{2p}{p-1}(\partial B(k+1))}^2\biggr).
\end{align*}
For $d\geq3$ let $p_*\geq1$ be such that $\frac1{p_*}=\frac1{d-1}+\frac12-\frac{1}{2p}$. Then a combination of Sobolev and H\"older inequality yield
\begin{align}\label{Lkey:step3:sobolev}
 \|u_t^{\alpha}\|_{L^\frac{2p}{p-1}(\partial B(k))}^2&\lesssim \|\nabla (u_t^{\alpha})\|_{L^{p_*}(\partial B(k))}^2+k^{-2}\|u_t^{\alpha}\|_{L^{p_*}(\partial B(k))}^2\notag\\
&\lesssim k^{2-(d-1)\frac{1}p}\|\omega^{-1}\|_{\underline L^q(\partial B(k))}\|\sqrt\omega \nabla u_t^{\alpha}\|_{L^2(\partial B(k))}^2+k^{-(d-1)\frac{1}p}\|u_t^{\alpha}\|_{L^2(\partial B(k))}^2
\end{align}
where the last inequality is valid since $\frac1q+\frac{1}p\leq\frac2{d-1}$ and $k^{d-1}\lesssim |\partial B(k)|\lesssim k^{d-1}$. Choosing $\delta=(\frac{1}p+\frac1q+1)^{-1}$, we obtain as in \eqref{est:lemmakeystep2almostfinal}
\begin{align}\label{est:lemmakeystep2almostfinals4}
&(m-n)^{2+\frac1p+\frac1q}\min_{\eta\in\mathcal A(m,n)}\int_{I_2}\sum_{\ee\in\mathbb B^d}\omega(\ee)(\nabla \eta(\ee))^{2}(u_t^\alpha(\ee))^{2}\dt\notag\\
\lesssim& \biggl(\sum_{k=n}^{m-1}  \|\omega\|_{\underline L^p(S(k))}^{\delta}\biggl(\int_{I_2} \sum_{i=k}^{k+1}i^2\|\omega^{-1}\|_{\underline L^q(\partial B(i))}\|\sqrt\omega \nabla (u^{\alpha})\|_{L^2(\partial B(i))}^2+\|u^{\alpha}\|_{L^2(\partial B(i))}^2\dt\biggr)^\delta\biggr)^\frac1\delta\notag\\
\leq&\biggl(\sum_{k=n}^{m-1}\|\omega\|_{\underline L^p(S(k))}^{p}\biggr)^\frac1p\notag\\
&\times \biggl(\sum_{k=n}^{m-1}\biggl(\int_{I_2} \sum_{i=k}^{k+1}i^2\|\omega^{-1}\|_{\underline L^q(\partial B(i))}\|\sqrt\omega \nabla (u^{\alpha})\|_{L^2(\partial B(i))}^2+\|u^{\alpha}\|_{L^2(\partial B(i))}^2\dt\biggr)^\frac{q}{q+1}\biggr)^{\frac{q+1}{q}}.
\end{align}
Combining \eqref{est:lemmakeystep2almostfinals4} with H\"older inequality in the form \eqref{est:lemmakeystep2almostfinal1} (with $\nu$ replaced by $1$), we obtain \eqref{est:L:key:s2:claima:simpler}.

For $d=2$, we argue as above but replace \eqref{Lkey:step3:sobolev} by
\begin{align*}
\|u_t^{\alpha}\|_{L^\frac{2p}{p-1}(\partial B(k))}^2\lesssim& k^{1-\frac1p}\|u_t^\alpha\|_{L^\infty(\partial B(k))}^2\\
\lesssim& k^{1-\frac1p}(\|\nabla (u_t^{\alpha})\|_{L^{1}(\partial B(k))}^2+k^{-2}\|u_t^{\alpha}\|_{L^{1}(\partial B(k))}^2)\\
\lesssim& k^{2-\frac{1}p}\|\omega^{-1}\|_{\underline L^q(\partial B(k))}\|\sqrt\omega \nabla (u_t^{\alpha})\|_{L^2(\partial B(k))}^2+k^{-\frac{1}p}\|u_t^{\alpha}\|_{L^2(\partial B(k))}^2.
\end{align*}

\ER

\ER

\substep{\ref{lm2:step2}.2}
We claim that there exists $c_2=c_2(d,p,q)\in[1,\infty)$ such that 
\begin{equation}\label{est:lkeyss2}
 \int_{I_2}\|u_t^{\alpha}\|_{L^{2(1+\e)}(B(m))}^{2(1+\e)}\dt
  \le  c_2 m^2 s_2|B(m)|\biggl(\sup_{t\in I_2}\|u_t^{\alpha}\|_{\underline L^2(B(m))}^2\biggr)^{\delta_2}\|\omega^{-1}\|_{\underline L^q(B(m))} \|u^{\alpha\nu}\|_{\underline H^{1}(I_2\times B(m))}^2.
\end{equation}

\smallskip

Let $Q>2$ be the Sobolev exponent for $\frac{2q}{q+1}$ in $\mathbb R^d$ given by 
\begin{equation}\label{def:Q}
\frac1Q=\frac1{2}-\frac12(\frac2d-\frac{1}{q})\stackrel{\eqref{def:delta}}{=}\frac{1-\delta_2}2
\end{equation}
(recall $\delta_2=\frac2d-\frac1q\in(0,1)$). Recalling $\e=\frac{\delta_2}\theta$ and $\nu=1-\delta_2(1-\frac1\theta)=1-\delta_2+\e$ (see \eqref{def:eps}), we obtain
\begin{equation}\label{est:nuQ}
\nu Q-2(1+\e)=2\e\biggl(\frac1{1-\delta_2}-1\biggr)>0
\end{equation}
and the interpolation inequality 
$$
 \|v\|_{L^{2(1+\e)}}\leq \|v\|_{L^{2}}^\ell\|v\|_{L^{\nu Q}}^{1-\ell}
$$
with
\begin{equation}\label{def:interpolationell}
\frac1{2(1+\e)}=\frac\ell2+\frac{1-\ell}{\nu Q}\qquad\biggl(\mbox{and thus}\quad\ell=\frac{\frac1{2(1+\e)}-\frac1{\nu Q}}{\frac12-\frac1{\nu Q}} \biggr)
\end{equation}
implies
\begin{align}\label{est:lkeyss21}
 \|u_t^{\alpha}\|_{L^{2(1+\e)}(B(m))}^{2(1+\e)} \leq \|u_t^\alpha\|_{L^2(B(m))}^{2(1+\e)\ell}\|u_t^\alpha\|_{L^{\nu Q}(B(m))}^{2(1+\e)(1-\ell)}=\|u_t^\alpha\|_{L^2(B(m))}^{2\delta_2}\|u_t^{\alpha\nu}\|_{L^{Q}(B(m))}^{2} 
\end{align}
where in the last relation we used
\begin{align}\label{eq:ellepsnu}
\frac{1+\e}{\nu }(1-\ell)=\frac{1+\e}\nu\frac{1-\frac1{1+\e}}{1-\frac2{\nu Q}}=\frac{\e}{\nu-\frac2Q}=1,
\end{align}
and thus
\begin{equation*}
 (1+\e)\ell\stackrel{\eqref{eq:ellepsnu}}{=}1+\e-\nu\stackrel{\eqref{def:eps}}{=}\delta_2.
\end{equation*}
Since $\frac1Q=\frac1{2}-\frac12(\frac2d-\frac{1}{q})$, a combination of Sobolev and H\"older inequality yields
\begin{equation}\begin{aligned}\label{est:lkeyss22}
 \|u_t^{\alpha\nu}\|_{\underline L^Q(B(m))}^2 &\lesssim \left(m\|\nabla (u_t^{\alpha\nu})\|_{\underline L^{\frac{2q}{q+1}}(B(m))}+\|u_t^{\alpha\nu}\|_{\underline L^{\frac{2q}{q+1}}(B(m))}\right)^{2}\\
 &\le 2\left(m^2\|\omega^{-1}\|_{\underline L^q(B(m))}\|\sqrt{\omega}\nabla  (u_t^{\alpha\nu})\|_{\underline L^{2}(B(m))}^2+\|u_t^{\alpha \nu}\|_{\underline L^2(B(m))}^2\right).
\end{aligned}\end{equation}
Combining \eqref{est:lkeyss21} and \eqref{est:lkeyss22}, we obtain
\begin{align*}
\|u_t^{\alpha}\|_{L^{2(1+\e)}(B(m))}^{2(1+\e)}\lesssim&|B(m)|\|u_t^{\alpha}\|_{\underline L^2(B(m))}^{2\delta_2}\left(m^2\|\omega^{-1}\|_{\underline L^q(B(m))}\|\sqrt{\omega}\nabla  (u_t^{\alpha\nu})\|_{\underline L^{2}(B(m))}^2+\|u_t^{\alpha \nu}\|_{\underline L^2(B(m))}^2\right),
\end{align*}
and the claimed estimate \eqref{est:lkeyss2} follows by integration in time.

\substep{\ref{lm2:step2}.3} Proof of \eqref{L1:estkey}. A direct consequence of \eqref{est:L:key:s2:claima} and \eqref{est:lkeyss2} is 
\begin{align*}
&\min_{\eta\in\mathcal A(n,m)}\int_{I_2}\sum_{\ee\in\mathbb B^d}\omega(\ee)(\nabla \eta(\ee))^{2}(u_t^{\alpha}(\ee))^{2}\dt\notag\\
\leq&c_1c_2^{1-\frac1\theta}\biggl(\sup_{t\in I_2}\|u_t^{\alpha}\|_{\underline L^2(B(m))}^2\biggr)^{\delta_2(1-\frac1\theta)}\frac{\|\omega\|_{\underline L^p(B(m))}\|\omega^{-1}\|_{\underline L^q(B(m))}|B(m)||I_2|\|u^{\alpha\nu}\|_{\underline H^{1}(I_2\times B(m))}^2}{\left(1-\frac{n}m\right)^{2+\frac1p+\frac1{\theta q}}},
\end{align*}
which implies the claimed estimate \eqref{L1:estkey} 
(using $m\leq 2n$, $1-\nu\stackrel{\eqref{def:nu}}{=}\delta_2(1-\frac1\theta)$, and $\gamma = 2 + \frac 2 p + \frac{1}{\theta q}$).

\substep{\ref{lm2:step2}.4} Proof of \eqref{ineq:interpolationrange}.

For $d=2$ (and thus $\theta=p$) \eqref{ineq:interpolationrange} reads 
$$
 0<\frac{p-1+(1-\frac1q)(1-\frac1p)}{p-1}<\frac{p}{p-1}<\infty
$$
which is obviously true since $p,q>1$. 

Let us now verify \eqref{ineq:interpolationrange} for $d\geq3$. The last inequality in \eqref{ineq:interpolationrange} is trivial, while the first follows directly from $0<\nu<1<\theta<p$ since $d\geq3$. Next, we observe that
\begin{align*}
 &\frac{p}{p-1}-\frac{\theta-\nu}{\theta-1}\stackrel{\eqref{def:nu}}{=}\frac1{p-1}-\frac{\delta_2}\theta=\frac{\theta-\delta_2(p-1)}{(p-1)\theta},\\
 &\frac{p\nu}{p-\theta}-\frac{p}{p-1}=\frac{p}{p-\theta}\biggl(\nu-1+\frac{\theta-1}{p-1}\biggr)\stackrel{\eqref{def:nu}}{=}\frac{p}{p-\theta}\biggl(\frac{\theta-1}{p-1}-\delta_2\frac{\theta-1}\theta\biggr)=\frac{p(\theta-1)}{(p-\theta)(p-1)\theta}(\theta-\delta_2(p-1)),
\end{align*}
and thus the second and third inequality in \eqref{ineq:interpolationrange} are equivalent to $\theta-\delta_2(p-1)>0$. In the case $d\geq3$, we have
\begin{align*}
 \theta-\delta_2(p-1)\stackrel{\eqref{def:delta}}{=}\frac{2p}{d-1}-\frac{p}q-(\frac2d-\frac1q)(p-1)=\frac2d-\frac1q+2p(\frac1{d-1}-\frac1d)\geq \frac2d-\frac1q>0,
\end{align*}
where the last inequality follows from the assumption $q>\frac2d$. Hence, the inequalities \eqref{ineq:interpolationrange} are proven.

\PfStep{lm2}{} 
Proof of estimates \eqref{est:keysubsol} and \eqref{est:keysubsol3}. Estimate \eqref{est:keysubsol3} follows directly from \eqref{est:basic:sub} and \eqref{est:L:key:s2:claima:simpler}. To show \eqref{est:keysubsol}, we combine \eqref{est:basic:sub} and \eqref{L1:estkey} to obtain
\begin{align}\label{est:keysubsolalmost1}
&\sup_{t\in I_1}\|u_t^\alpha\|_{\underline L^2(B(n))}^2 +\int_{I_1} \|\sqrt\omega\nabla (u_t^\alpha)\|_{\underline L^2(B(n))}^2\notag\\
&\leq c\alpha^2\biggl(\sup_{t\in I_2}\|u^\alpha\|_{\underline L^2(B(m))}^2\biggr)^{1-\nu}\|\omega\|_{\underline L^p(B(m))}\|\omega^{-1}\|_{\underline L^q(B(m))}\frac{s_2}{(1-\tfrac{n}m)^{\gamma}}\|u^{\alpha\nu}\|_{\underline H^{1}(I_2\times B(m))}^2\notag\\
&\quad + c (s_2-s_1)^{-1}\int_{I_2}\|u_t^\alpha\|_{\underline L^2(B(m))}^2\dt,
\end{align}
where $c=c(d,p,q)\in[1,\infty)$. The first term on the right-hand side has already the desired form, hence we only need to estimate the second term: 
Let $Q>2$ be the Sobolev exponent for  $\frac{2q}{q+1}$ given by \eqref{def:Q}, which by~\eqref{est:nuQ} satisfies $\nu Q > 2(1+\e) > 2$. Combination of Jensen and Sobolev inequality yield
\begin{align}\label{est:keysubsolalmost2}
\int_{I_2}\|u_t^\alpha\|_{\underline L^2(B(m))}^2\dt &\leq \left(\sup_{t\in I_2}\|u_t^\alpha\|_{\underline L^2(B(m))}^2\right)^{1-\nu}\int_{I_2}\|u_t^\alpha\|_{\underline L^2(B(m))}^{2\nu}\dt\notag\\
&\leq\left(\sup_{t\in I_2}\|u_t^\alpha\|_{\underline L^2(B(m))}^2\right)^{1-\nu}\int_{I_2}\|u_t^{\alpha}\|_{\underline L^{\nu Q}(B(m))}^{2\nu}\dt\notag\\ 
&=  \left(\sup_{t\in I_2}\|u_t^\alpha\|_{\underline L^2(B(m))}^2\right)^{1-\nu}\int_{I_2}\|u_t^{\alpha \nu}\|_{\underline L^{Q}(B(m))}^{2}\dt\notag\\ 
&\lesssim\left(\sup_{t\in I_2}\|u_t^\alpha\|_{\underline L^2(B(m))}^2\right)^{1-\nu}\int_{I_2}m^2\|\nabla (u_t^{\alpha\nu})\|_{\underline L^{\frac{2q}{q+1}}(B(m))}^{2}+\|u_t^{\alpha\nu}\|_{\underline L^{\frac{2q}{q+1}}(B(m))}^{2}\dt\notag\\
&\leq\left(\sup_{t\in I_2}\|u_t^\alpha\|_{\underline L^2(B(m))}^2\right)^{1-\nu}\int_{I_2}m^2\|\omega^{-1}\|_{\underline L^q(B(m))}\|\sqrt{\omega}\nabla (u_t^{\alpha\nu})\|_{\underline L^{2}(B(m))}^{2}+\|u_t^{\alpha\nu}\|_{\underline L^{2}(B(m))}^{2}\dt\notag\\
&=\left(\sup_{t\in I_2}\|u_t^\alpha\|_{\underline L^2(B(m))}^2\right)^{1-\nu}s_2 m^2\|\omega^{-1}\|_{\underline L^q(B(m))}\|u^{\alpha\nu}\|_{\underline H^1(I_2\times B(m))}^2.
\end{align}
Recall that the $H^1$-norm consist of the $L^2$ norm of the function and its gradient (see~\eqref{def:h1}), so to get the full $H^1$-norm on the left-hand side of~\eqref{est:keysubsolalmost1} we need to add and estimate $u_t^{\alpha}$ itself: for that we observe that the assumption $1\leq \frac{m}{n}\leq 2$ and \eqref{est:keysubsolalmost2},  applied on $B(n)$ instead of $B(m)$, yield
\begin{equation}
\begin{aligned}\label{est:keysubsolalmost3}
|B(n)|^{-\frac{2}d}\|\omega^{-1}\|_{\underline L^q(B(n))}^{-1}\int_{I_1}\|u_t^\alpha\|_{\underline L^2(B(n))}^2\dt &\lesssim\left(\sup_{t\in I_1}\|u_t^\alpha\|_{\underline L^2(B(n))}^2\right)^{1-\nu} s_1 \|u^{\alpha\nu}\|_{\underline H^1(I_1\times B(n))}^2 \\
&\lesssim\left(\sup_{t\in I_2}\|u_t^\alpha\|_{\underline L^2(B(m))}^2\right)^{1-\nu}s_2 \|u^{\alpha\nu}\|_{\underline H^1(I_2\times B(m))}^2.
\end{aligned}
\end{equation}
Estimate \eqref{est:keysubsol} follows from \eqref{est:keysubsolalmost1}--\eqref{est:keysubsolalmost3} and $1\leq \|\omega\|_{\underline L^p(B(m))}\|\omega^{-1}\|_{\underline L^q(B(m))}$.

\end{proof}

Next, we combine Lemma~\ref{L:key} with a variation of Moser iteration method to prove local boundedness of non-negative subcaloric functions. For this, we apply the estimates of Lemma~\ref{L:key} on a sequence of parabolic cylinder. Fix $\tau>0$ and let $x_0\in \mathbb Z^d$, $t_0\in\mathbb R$, $n\geq0$ and $\sigma\in(0,1]$. Set
\begin{equation*}
 Q_\sigma(t_0,x_0,\tau,n):=[t_0-\sigma \tau n^2,t_0]\times B(x_0,\sigma n).
\end{equation*}
\begin{theorem}[Local boundedness]\label{T:bound}
Fix $d\geq2$, $\omega\in\Omega$, $p\in(1,\infty)$ and $q\in(\frac{d}2,\infty)$ satisfying $\frac1p+\frac1q<\frac2{d-1}$. 
There exists $c=c(d,p,q)\in[1,\infty)$ such that the following is true: Let $u>0$ be a subcaloric function in $Q_1(t_0,x_0,\tau,n)$ with $t_0\in\R$, $x_0\in\mathbb Z^d$, $n\ge 2^{20}$, and $\tau>0$. Then
\begin{equation}\label{est:Tbound}
\|u\|_{L^\infty(Q_\frac12(t_0,x_0,\tau,n))}\leq c\mathcal C(\omega,B(x_0,n),\tau)^\frac{p}{p-1}\|u\|_{\underline L^{2}(Q_1(t_0,x_0,\tau,n))}, 
\end{equation}
where 
\begin{equation}\label{def:C}
\mathcal C(\omega,B,\tau):=\max\{1,\tau^\frac12\biggl(\|\omega^{-1}\|_{\underline L^q(B)}(\|\omega\|_{\underline L^p(B)}+\tau^{-1})^{2-\nu}\biggr)^{\frac1{2(1-\nu)}}\}\in[1,\infty)
\end{equation}
and $\nu=\nu(d,p,q)\in(0,1)$ is given in \eqref{def:nu}.

\end{theorem}

\begin{remark}\label{rem:Tboundpqoptimal}
In recent works \cite{CS19,AT19} the statement of Theorem~\ref{T:bound} is proven under the more restrictive relation $\frac1p+\frac1q<\frac2d$. The restrictions on $p$ and $q$ in Theorem~\ref{T:bound} are essentially optimal: Counterexamples to elliptic regularity in the form \cite[Theorem~2.6]{BK14} show that local boundedness in the form \eqref{est:Tbound} with \eqref{def:C} fails already for $\mathcal L^\omega$-harmonic functions if $\frac1p+\frac1q>\frac2{d-1}$, see \cite[Remark 2 and 4]{BS19b} for a more detailed discussion. The additional restriction on $q$, namely $q>\frac{d}2$, is not present in the corresponding elliptic version of Theorem~\ref{T:bound} (see \cite[Theorem~2]{BS19b} and Theorem~\ref{T:ellipticboundedness} below) and can be related to trapping phenomena for random walks in random environments. In Remark~\ref{rem:trapping} below we discuss this in more detail and show that local boundedness in the form of Corollary~\ref{C:Tbound} below (which is a direct consequence of Theorem~\ref{T:bound}) is not valid for $q<\frac{d}2$. 
\end{remark}


\begin{proof}[Proof of Theorem \ref{T:bound}]

Without loss of generality we consider $t_0=0$ and $x_0=0$, and we use the shorthand $Q_\sigma(\tau,n)=Q_\sigma(0,0,\tau,n)$. Throughout the proof we write $\lesssim$ if $\leq$ holds up to a positive constant that depends only on $d,p$ and $q$. The proof is divided in three steps: (i) using Lemma~\ref{L:key} and an iteration argument, we obtain a one-step improvement; (ii) the one-step improvement and a Moser iteration-type argument yield local boundedness in the form \eqref{est:Tbound} where the $L^2$-norm on the right-hand side is replaced by a slightly stronger norm of $u$; (iii) finally a well-known interpolation argument yield the claimed estimate. 

\step 1 One-step improvement.

Let $u>0$ be subcaloric and $\alpha\geq1$. We claim that there exists $c=c(d,p,q)\in[1,\infty)$ such that for all $\frac12\leq \rho<\sigma<1$ and $n(\sigma-\rho)\geq  2^5$ it holds
\begin{align}\label{est:moser:iterate:prelimclaim}
 &\sup_{ t\in [-\rho \tau n^2, 0]}(\rho \tau n^{2})^{-1}\|u_t^\alpha\|_{\underline L^{2}(B(\rho n))}^2+\|u^{\alpha}\|_{\underline H^{1}(Q_\rho(\tau,n))}^{2}\notag\\
 &\quad \le \biggl(c|B(n)|^{(1-\nu)^{\lfloor \log_2((\sigma-\rho)n)\rfloor- 3}}\alpha^2(\tau n^2)^{1-\nu}\frac{\|\omega^{-1}\|_{\underline L^q(B(\sigma n))}}{(\sigma-\rho)^{\hat \gamma}} \big(\|\omega\|_{\underline L^p(B(\sigma n))}+\tau^{-1}\bigr)\|u^{\alpha \nu}\|_{\underline H^{1}(Q_\sigma(\tau, n))}^{2}\biggr)^\frac1{\nu},
\end{align}
where $\nu\in(0,1)$ is defined in \eqref{def:nu} and 
\begin{equation}\label{def:gammahat}
\hat \gamma=2+\frac1p+\frac1q\in[2,\infty).
\end{equation}
For $k\in\mathbb N\cup\{0\}$, set
\begin{equation}\label{def:sigmak}
 \sigma_k:=\frac12(\sigma+\rho)-2^{-{(k+1)}}(\sigma-\rho),\qquad B_k:=B(\sigma_k n),\qquad I_k:=[-\sigma_k \tau n^2,0].
\end{equation}
In view of Lemma~\ref{L:key}, there exists $c=c(d,p,q)\in[1,\infty)$ such that for any $k\in\mathbb N$ satisfying $2^{k+3}\leq (\sigma-\rho)n$ (which ensures  $\lfloor \sigma_{k+1} n\rfloor -\lfloor \sigma_{k} n\rfloor\geq1$)
\begin{align}\label{est:moser:iterate:prelim}
 &\sup_{t\in I_{k}}|I_k|^{-1}\|u_t^\alpha\|_{\underline L^{2}(B_k)}^2+\|u^{\alpha}\|_{\underline H^{1}(I_{k}\times B_k)}^{2}\notag\\
 &\le c \alpha^2 2^{(k+1){\hat \gamma}}\biggl(\sup_{t\in I_{k+1}}|I_{k+1}|^{-1}\|u_t^\alpha\|_{\underline L^2(B_{k+1})}^2\biggr)^{\!1-\nu} \!\!\!(\tau n^2)^{1-\nu}\notag\\
&\quad\times
\|\omega^{-1}\|_{\underline L^q(B(\sigma n))}\frac{\|\omega\|_{\underline L^p(B(\sigma n))}+\tau^{-1}}{(\sigma-\rho)^{{\hat \gamma}}}\|u^{\alpha\nu}\|_{\underline H^{1}(Q_\sigma(\tau, n))}^2.
 \end{align}
Indeed, \eqref{est:moser:iterate:prelim} follows from estimate \eqref{est:keysubsol} (with $n=\lfloor\sigma_k n\rfloor$, $m=\lfloor\sigma_{k+1}n\rfloor$, $s_1=\tau\sigma_kn^2$, and $s_2=\tau \sigma_{k+1}n^2$) and $\hat {\gamma}\geq\gamma\geq2$ (where $\gamma$ is given as in Lemma~\ref{L:key}), together with the elementary estimates 
\begin{align*}
\biggl(1-\frac{\lfloor \sigma_{k}n\rfloor}{\lfloor \sigma_{k+1}n\rfloor}\biggr)^{-\gamma}\leq&\left(\frac{n}{2^{-(k+2)}(\sigma-\rho)n-1}\right)^\gamma\leq \left(\frac{2^{k+3}}{\sigma-\rho}\right)^\gamma\\
\frac1{\sigma_{k+1}\tau n^2-\sigma_{k}\tau n^2}=&\frac{2^{k+2}}{n^2\tau(\sigma-\rho)} .
\end{align*}
Set $\hat k:=\hat k(n,\sigma-\rho):=\lfloor \log_2((\sigma-\rho)n)\rfloor- 3$, so that $\hat k=\max\{k\in\mathbb N\,|\,2^{k+3}\leq (\sigma-\rho)n\}$ and hence $m-n \ge 1$. 

Using \eqref{est:moser:iterate:prelim} $(\hat k-1)$-times, we obtain
\begin{align}\label{est:moser:almostfinal:prelim}
&\sup_{t\in I_0}|I_{0}|^{-1}\|u_t^\alpha\|_{\underline L^{2}(B_0)}^2+\|u^{\alpha}\|_{\underline H^{1}(I_{0}\times B_{0})}^{2}\notag\\
&\begin{aligned}
&\le \biggl(\sup_{t\in I_{\hat k-1}}|I_{\hat k-1}|^{-1}\|u_t^\alpha\|_{\underline L^{2}(B_{\hat k-1})}^2\biggr)^{(1-\nu)^{\hat k-1}}\left(2^{\hat \gamma\sum_{k=0}^{\hat k-2}(k+1)(1-\nu)^k}\right)\\
&\quad \times \biggl(c \alpha^2  (\tau n^2)^{1-\nu}\|\omega^{-1}\|_{\underline L^q(B(\sigma n))}\frac{\|\omega\|_{\underline L^p(B(\sigma n))}+\tau^{-1}}{(\sigma-\rho)^{\gamma}}\|u^{\alpha\nu}\|_{\underline H^{1}(Q_\sigma(\tau,n))}^2\biggr)^{\sum_{k=0}^{\hat k-2}(1-\nu)^k}.
\end{aligned}
\end{align}
We will repeatedly use discrete $\ell^r-\ell^s$ inequality for sequences with $r \ge s$, which after taking averages has the following form
\begin{equation}\label{lplq}
 |B(n)|^{\frac 1r - \frac 1s} \|f\|_{\underline L^r(B(n))} \le \|f\|_{\underline L^s(B(n))}.
\end{equation}

We estimate the first factor of the right-hand side in \eqref{est:moser:almostfinal:prelim} using~\eqref{lplq} and \eqref{est:keysubsol3}
\begin{align*}
 &\biggl(\sup_{t\in I_{{\hat k-1}}}|I_{\hat k-1}|^{-1}\|u_t^\alpha\|_{\underline L^{2}(B_{\hat k-1})}^2\biggr)^{(1-\nu)^{\hat k-1}}
 \\
 &\uoset{\eqref{est:keysubsol3}}{}{\le} \left(|B_{\hat k-1}| |I_{\hat k-1}|\right)^{\frac{(1-\nu)^{\hat k}}\nu}\biggl(\sup_{t\in I_{{\hat k-1}}}|I_{\hat k-1}|^{-1}\|u_t^{\alpha\nu}\|_{\underline L^{2}(B_{\hat k-1})}^2\biggr)^{\frac{(1-\nu)^{\hat k-1}}\nu}
 \\
&\stackrel{\eqref{est:keysubsol3}}{\le} \biggl(c\alpha^2 |B(n)|^{1-\nu} (\tau n^2)^{1-\nu}\|\omega^{-1}\|_{\underline L^q(B(\sigma n))}\frac{\|\omega\|_{\underline L^p(B(\sigma n))}+\tau^{-1}}{(\sigma-\rho)^{\hat\gamma}}\|u^{\alpha\nu}\|_{\underline H^1(Q_\sigma(\tau,n))}^2\biggr)^{\frac{(1-\nu)^{\hat k-1}}\nu}.
\end{align*}
The previous formula,~\eqref{est:moser:almostfinal:prelim}, and the identity $\sum_{k=0}^{\hat k-2}(1-\nu)^k=\frac{1-(1-\nu)^{\hat k-1}}{\nu}$ imply \eqref{est:moser:iterate:prelimclaim},  where we also used that 
$\sum_{k=0}^{\hat k-2}(k+1)(1-\nu)^k \le \frac{1}{\nu^2} \le c(p,q,d)$.

\step 2 Iteration.

We claim that there exists $c=c(d,p,q)\in[1,\infty)$ such that for all $n(\sigma-\rho)\ge  2^9$ holds
\begin{align}\label{est:Tboundp}
&\|u\|_{L^\infty(Q_\rho(\tau,n))}\leq c|B(n)|^{\beta((\sigma-\rho)n)} \frac{\mathcal C(\omega,B(\sigma n),\tau)}{(\sigma-\rho)^{\frac{\hat\gamma}{2(1-\nu)}+1}}\|u\|_{2,2p',Q_\sigma(\tau, n)},
\end{align}
where $\mathcal C$ is defined in \eqref{def:C}, $p'=\frac{p}{p-1}$, and 
\begin{align}\label{def:beta}
\beta(z):=\frac{1}{2\nu} \bigl( \nu^{\frac12 \log_2(z) - 1} + (1-\nu)^{\frac12 \log_2(z) - 4} \bigr).
\end{align}
and for $I\subset \R$ and $m\in\mathbb N$, we set
\begin{equation*}
 \|v\|_{2,2p',I\times B(m)}:=\biggl(\fint_I \|v_t\|_{\underline L^{2p'}(B(m))}^2\dt\biggr)^\frac12.
\end{equation*}
For $k\in\mathbb N\cup\{0\}$ and $\frac 12 \le \rho < \sigma < 1$, set 
\begin{equation*}
 \alpha_k:=\nu^{-k+1}, \quad \sigma_k := \rho + 2^{-k}(\sigma-\rho), \quad B_k := B(\sigma_k n), \quad I_k := [-\sigma_k \tau n^2, 0].
\end{equation*}
%
Using~\eqref{est:moser:iterate:prelimclaim} from Step~1 with $\sigma_k$ and $\sigma_{k-1}$ playing role of $\rho$ and $\sigma$, respectively, we get 
\begin{align*}
 &\sup_{ t\in I_k} |I_k|^{-1} \|u_t^\alpha\|_{\underline L^{2}(B_k)}^2+\|u^{\alpha}\|_{\underline H^{1}(I_k \times B_k)}^{2}\notag\\
 &\quad \le \biggl(c|B(n)|^{(1-\nu)^{\lfloor \log_2((2^{-k}(\sigma-\rho)n)\rfloor- 3}}\alpha^2(\tau n^2)^{1-\nu}\frac{\|\omega^{-1}\|_{\underline L^q(B(\sigma n))}}{(2^{-k}(\sigma-\rho))^{\hat \gamma}} \big(\|\omega\|_{\underline L^p(B(\sigma n))}+\tau^{-1}\bigr)\|u^{\alpha \nu}\|^2_{\underline H^{1}(I_{k-1}\times B_{k-1}}\biggr)^\frac1{\nu},
\end{align*}
where $c=c(d,p,q)\in[1,\infty)$ and we required $n(\sigma-\rho)2^{-k} = n(\sigma_{k-1} - \sigma_k) \ge 2^5$. Observe that $\sigma/2 \le \sigma_k \le \sigma$ allowed us to replace the norms of $\omega$ on $B_k = B(\sigma_k n)$ with the ones on the larger ball $B(\sigma n)$ (while increasing $c$ by a fixed factor).

Using the last relation with $\alpha_k$ playing the role of $\alpha$ and afterwards taking both sides to the power $\frac{1}{2\alpha_k}$, yields \ER
\begin{align}\label{est:moser:iterate}
 &\sup_{t\in I_{k}}|I_k|^{-\frac1{2\alpha_{k}}}\|u_t\|_{\underline L^{2\alpha_{k}}(B_k)}+\|u^{\alpha_{k}}\|_{\underline H^{1}(I_{k}\times B_k)}^{\frac{1}{\alpha_{k}}}\notag\\
 &\le \biggl(c|B(n)|^{(1-\nu)^{\lfloor \log_2((\sigma-\rho)n)\rfloor- 3-k}} \nu^{-2(k-1)}(\tau n^2)^{1-\nu}\frac{\|\omega^{-1}\|_{\underline L^q(B(\sigma n))}}{(\sigma-\rho)^{\hat \gamma}} 2^{\hat \gamma k}\big(\|\omega\|_{\underline L^p(B(\sigma n))}+\tau^{-1}\bigr)\biggr)^\frac1{2\alpha_{k-1}}\notag\\
 &\quad \times\|u^{\alpha_{k-1}}\|_{\underline H^{1}(I_{k-1}\times B_{k-1})}^{\frac{1}{\alpha_{k-1}}},
\end{align}
where we used $\alpha_k \nu = \alpha_{k-1}$. Fix $\hat k\leq \hat k_1(n,\sigma-\rho):=\lfloor \log_2((\sigma-\rho)n)\rfloor -5=\max\{k\in\mathbb N\,|\,2^{k+5}\leq (\sigma-\rho)n\}$. Using \eqref{est:moser:iterate} $(\hat k-1)$-times, we obtain
\begin{align}\label{est:moser:almostfinal}
&\sup_{t\in I_{{\hat k}}}\|u_t\|_{\underline L^{2\alpha_{\hat k}}(B_{\hat k})}\notag\\
&\quad \le \prod_{k=2}^{\hat k}\biggl(c|B(n)|^{(1-\nu)^{\lfloor \log_2((\sigma-\rho)n)\rfloor-3-k}}\nu^{-2(k-1)}(\tau n^2)^{1-\nu}\frac{\|\omega^{-1}\|_{\underline L^q(B(\sigma n))}}{(\sigma-\rho)^{\hat\gamma}}2^{\hat\gamma k}(\|\omega\|_{\underline L^p(B(\sigma n))}+\tau^{-1})\biggr)^\frac1{2\alpha_{k-1}}\notag\\
&\qquad \times  |I_{{\hat k}}|^\frac1{2\alpha_{\hat k}}\|u\|_{\underline H^{1}(I_1\times B_1)}\notag\\
 &\quad \le \tau^\frac12n\biggl(c\frac{\|\omega^{-1}\|_{\underline L^q(B(\sigma n))}}{(\sigma-\rho)^{\hat\gamma}}(\|\omega\|_{\underline L^p(B(\sigma n))}+\tau^{-1})\biggr)^{\frac1{2}\sum_{k=0}^{\hat k-2} \nu^{k}}|B(n)|^{\frac12(1-\nu)^{\lfloor \log_2((\sigma-\rho)n)\rfloor-5}\sum_{k=0}^{\hat k-2}(\frac\nu{1-\nu})^k}\notag\\
 &\qquad \times\left(2^{\hat\gamma}\nu^{-2}\right)^{\frac1{2}\sum_{k=0}^\infty(k+2)\nu^{k}}\|u\|_{\underline H^{1}(I_1\times B_1)},
\end{align}
where we used $(1-\nu) \sum_{k=0}^{\hat k-2}\nu^k=(1-\nu) \frac{1-\nu^{\hat k-1}}{1-\nu} = 1-\nu^{\hat k-1}$ and $|I_{\hat k}|^{\frac{1}{2\alpha_{\hat k}}} \le (\tau n^2)^{\nu^{\hat k-1}/2}$
to deal with the $(\tau n^2)$-term\ER. To estimate the right-hand side of \eqref{est:moser:almostfinal}, we use \eqref{est:keysubsol2}, Jensen's inequality, and $\|\omega^{-1}\|_{\underline L^q(B_1)} \|\omega\|_{\underline L^p(B_1)} \ge 1$ to get
\begin{align*}
\|\sqrt{\omega}\nabla u\|_{\underline L^{2}(I_1\times B_1)}^2
&\stackrel{\eqref{est:keysubsol2}}{\lesssim}\frac{\|\omega\|_{\underline L^p(B(\sigma n))}+\tau^{-1}}{n^2(\sigma-\rho)^2} \|u \|_{2,2p',Q_\sigma(\tau,n)}^2
\\
| B_1|^{-\frac2d}\|\omega^{-1}\|_{\underline L^q(B_1)}^{-1}\|u\|_{\underline L^{2}(I_1\times B_1)}^2
&\uoset{\eqref{est:keysubsol2}}{}{\lesssim}
n^{-2}\|\omega\|_{\underline L^p(B(\sigma n))}\|u\|_{2,2p',Q_\sigma(\tau, n)}^2,
\end{align*}
and thus there exists $c=c(d,p,q)\in[1,\infty)$ such that
\begin{equation}\label{est:uh1ul22p}
\|u\|_{\underline H^{1}(I_1\times B_1)}\leq \biggl(c\frac{\|\omega\|_{\underline L^p(B(\sigma n))}+\tau^{-1}}{n^2(\sigma-\rho)^2}\biggr)^\frac1{2}\|u\|_{2,2p',Q_\sigma(\tau, n)}.
\end{equation}
Since $\|\omega\|_{\underline L^p(S)}\|\omega^{-1}\|_{\underline L^q(S)}\geq1$ for any $S\subset\mathbb B^d$, $\sum_{k=0}^\infty (1+k) \nu^{-k}\lesssim1$ and 
\begin{equation*}
|B(n)|^{\frac12(1-\nu)^{\lfloor \log_2((\sigma-\rho)n)\rfloor -5}\sum_{k=0}^{\hat k-2}(\frac\nu{1-\nu})^k}\leq |B(n)|^{\frac1{2}(1-\nu)^{\lfloor \log_2((\sigma-\rho)n)\rfloor -5}\frac1\nu(1-\nu)^{2-\hat k}},
\end{equation*}
which follows from $\sum_{k=0}^{\hat k-2} \bigl( \frac{\nu}{1-\nu}\bigr)^k \le \sum_{k=0}^{\hat k-2} \bigl( \frac{1}{1-\nu}\bigr)^k = \frac{(1-\nu)^{1-\hat k} - 1}{(1-\nu)^{-1}-1} \le \nu^{-1} (1-\nu)^{2-\hat k}$, \ER
we obtain with the choice $\hat k=\lfloor \frac12 \log_2((\sigma-\rho)n)\rfloor$, which thanks to $(\sigma-\rho)n \ge 2^9$ satisfies the necessary condition $\hat k \le \hat k_1$, that
\begin{align*}
 &\|u\|_{L^\infty(Q_\rho(n))}\notag\\
 &\uoset{\eqref{lplq}}{}{=}\sup_{t\in[-\rho \tau n^2,0]}\|u_t\|_{L^\infty(B(\rho n))}\notag
 \\
 &\uoset{}{\eqref{lplq}}{\le} |B_{\hat k}|^\frac1{2\alpha_{\hat k}}\sup_{t\in I_{\hat k}}\|u_t\|_{\underline L^{2\alpha_{\hat k}}(B_{\hat k})}\notag
 \\
 &\stackrel{\eqref{est:moser:almostfinal}}{\lesssim}|B(n)|^{\frac12\nu^{\hat k-1}+\frac1{2\nu}(1-\nu)^{\lfloor \log_2((\sigma-\rho)n)\rfloor-3-\hat k}}\tau^\frac12 n\biggl(\frac{\|\omega^{-1}\|_{\underline L^q(B(\sigma n))}}{(\sigma-\rho)^{\hat\gamma}}(\|\omega\|_{\underline L^p(B(\sigma n))}+\tau^{-1})\biggr)^{\frac1{2(1-\nu)}}\|u\|_{\underline H^{1}(I_1\times B_1)}\notag
 \\
 &\stackrel{\eqref{est:uh1ul22p}}{\lesssim}
 |B(n)|^{\beta((\sigma-\rho)n)}\tau^\frac12\biggl(\frac{\|\omega^{-1}\|_{\underline L^q(B(\sigma n))}}{(\sigma-\rho)^{\hat\gamma+2(1-\nu)}}(\|\omega\|_{\underline L^p(B(\sigma n))}+\tau^{-1})^{2-\nu}\biggr)^{\frac1{2(1-\nu)}}\|u\|_{2,2p',Q_\sigma(\tau, n)}
\end{align*}
which proves the claim. Since $\nu \in (0,1)$, the last inequality follows from $\nu^{\hat k} \le \nu^{\frac12 \log_2((\sigma-\rho)n)-1}$ and $(1-\nu)^{\lfloor \log_2((\sigma-\rho)n)\rfloor - 3 - \hat k} \le (1-\nu)^{\frac12 \log_2((\sigma-\rho)n) - 4}$.

\step 3 Conclusion

For $k\in\mathbb N\cup\{0\}$, we set 
\begin{equation}\label{def:sigmakcor}
 \sigma_k:=\frac34-\frac1{4^{1+k}}\qquad B_k:=B(\sigma_k n),\qquad I_k:=[-\sigma_k \tau n^2,0].
\end{equation}
Combining the interpolation inequality 
\begin{equation}\label{ineq:esttlocalboundinterpolation}
 \|u\|_{\underline L^{2p'}(I_k \times B_k)}\leq \|u\|_{\underline L^{2}(I_k\times B_k)}^\frac1{p'} \|u\|_{L^\infty(I_k\times B_k)}^{\frac1{p}}
\end{equation}
(where $p'=\frac{p}{p-1}$) estimate \eqref{est:Tboundp} (with $\sigma=\sigma_k$ and $\rho=\sigma_{k-1}$) and Jensen inequality in the form $\|v\|_{2,2p',I\times B}\leq \|v\|_{\underline L^{2p'}(I\times B)}$, we obtain for $ k\leq \frac12( \log_2(3n)- 11)$ (which ensures $(\sigma_k-\sigma_{k-1})n\ge 2^9$) and $\gamma' := \frac{\hat \gamma}{2(1-\nu)} + 1$ that
\begin{align}\label{est:iteration:step3Tbound}
\|u\|_{L^\infty(I_{k-1}\times B_{k-1})}&\lesssim 4^{(k+1) \gamma'}|B(n)|^{\beta(\frac{3n}{4^{k+1}} )}\mathcal C(\omega,B_k,\tau)\|u\|_{\underline L^{2p'}(I_k\times B_k)}\notag
\\
&\le  4^{(k+1) \gamma'}|B(n)|^{\beta(\frac{3n}{4^{k+1}} )}  M \|u\|_{\underline L^{2}([-\frac34 \tau n^2,0]\times B(\frac34 n))}^\frac1{p'} \|u\|_{L^\infty(I_k\times B_k)}^{\frac1{p}},
\end{align}
where
\begin{equation}\label{def:cc}
 M:=c\mathcal C(\omega,B(n),\tau)\in[1,\infty)
\end{equation}
with a suitable constant $c=c(d,p,q)\in[1,\infty)$. Iterating estimate \eqref{est:iteration:step3Tbound}, we obtain for every $\hat k \le  \frac12 (\log_2(3n)-11)$ that
\begin{align*}
&\|u\|_{L^\infty(I_0\times B_0)}\notag\\
&\le 4^{\gamma'  \sum_{k=0}^{\hat k-1}(k+2)p^{-k}}|B(n)|^{\sum_{k=0}^{\hat k -1}\beta(\frac{3n}{4^{k+2}})p^{-k}}( M  \|u\|_{\underline L^{2}(Q_\frac34(\tau,n))}^{1-\frac1p})^{\sum_{k=0}^{\hat k-1}p^{-k}}\|u\|_{ L^{\infty}(I_{\hat k}\times B_{\hat k})}^{p^{-\hat k}}.
\end{align*}
To estimate further the last factor on the right-hand side, we use \eqref{est:Tboundp} and the discrete estimate~\eqref{lplq} in the form $\|v\|_{\underline L^{2p'}(B(n))}\leq |B(n)|^\frac1{2p}\|v\|_{\underline L^{2}(B(n))}$:
\begin{equation}\label{est:Tboundalmostfinal2}
\|u\|_{ L^{\infty}(Q_{\frac34}(\tau,n))}\stackrel{\eqref{est:Tboundp}}{\lesssim}  |B(n)|^{\beta(\frac{n}4)}M\|u\|_{2,2p',Q_1(\tau,n)}\leq   |B(n)|^{\beta(\frac{n}4)+\frac1{2p}}M\|u\|_{\underline L^2(Q_1(\tau,n))}
\end{equation}
and thus we obtain
\begin{align}\label{est:Tboundalmostfinal1}
&\|u\|_{L^\infty(Q_{\frac12}(\tau,n))}\lesssim |B(n)|^{\sum_{k=0}^{\hat k -1}\beta(\frac{3n}{4^{k+2}})p^{-k}+\beta(\frac{n}4)p^{-\hat k}+\frac12 p^{-\hat k-1}}M^{\sum_{k=0}^{\hat k}p^{-k}} \|u\|_{\underline L^{2}(Q_{1}(\tau,n))},
\end{align}
where in this last inequality we used that $\sum_{k=0}^{\hat k-1} (k+2)p^{-k} \le \sum_{k=0}^{\infty} (k+2)p^{-k} \le c(p)$, which together with $\gamma' > 0$ gives $4^{ \gamma' \sum_{k=0}^{\hat k-1}(k+2)p^{-k}} \le c(d,p,q)$. 
Since $\sum_{k=0}^{\infty}\frac1{p^k}=\frac{1}{p-1}$, the claimed estimate \eqref{est:Tbound} follows from \eqref{est:Tboundalmostfinal1} provided we find $\hat k$ (depending on $n$) such that the prefactor on the right-hand side in \eqref{est:Tboundalmostfinal1} is uniformly bounded in $n$. Hence, it is left to find a sequence $(\hat k_n)_{n\in\mathbb N}\subset\mathbb N$ satisfying $\hat k_n \le \BR2 \frac12 (\log_2(3n)-11)$ for all $n$ sufficiently large such that 
\begin{equation}\label{est:limsupBn}
\limsup_{n\to\infty}|B(n)|^{\sum_{k=0}^{\hat k_n -1}\beta(\frac{3n}{4^{k+1}})p^{-k}+\beta(\frac{n}4)p^{-\hat k_n}+\frac12 p^{-\hat k_n-1}}\leq c(d,p,q)<\infty.
\end{equation}
First observe that $\beta(z) \stackrel{\eqref{def:beta}}{=} \frac{1}{2\nu} \bigl( \nu^{ \frac12 \log_2(z) - 1} + (1-\nu)^{\frac12 \log_2(z) - 4} \bigr) \le cz^{-\alpha}$ for some $c=c(d,p,q)$ and $\alpha=\alpha(d,p,q) > 0$. Indeed, \BR2 we have $\nu^{\frac12 \log_2(z)} = 2^{\frac12 \log_2(\nu) \log_2(z)} = z^{\frac12\log_2(\nu)}$ and thus $\frac{1}{2\nu} \nu^{\frac12\log_2(z)-1} \le cz^{-\alpha}$ follows from $\log_2(\nu) < 0$ (recall $\nu=\nu(d,p,q) \in(0,1)$). \ER The same argument works also for the $(1-\nu)$-part of $\beta$, hence giving the estimate $\beta(z) \le cz^{-\alpha}$. 
This estimate then implies 
%
\begin{align*}
 \sum_{k=0}^{\hat k-1}\beta(\frac{3n}{4^{k+1}}) p^{-k} &\le \sum_{k=0}^{\hat k-1}\beta(\frac{3n}{4^{k+1}})
 \le \frac{c}{n^\alpha} \sum_{k=0}^{\hat k - 1} 4^{k\alpha} \le c \biggl( \frac{4^{\hat k}}{n} \biggr)^{\alpha}. 
\end{align*}
The desired estimate \eqref{est:limsupBn} follows with the choice $\hat k_n:=\lfloor \frac14 \log_2(3n)\rfloor$, which in particular yields $\frac{4^{\hat k}}{n} \le cn^{-\frac12}$, and the elementary observation that for all $c_1,c_2,c_3>0$ and $\mu\in(0,1)$ we have
\begin{equation*}
 \limsup_{n\to\infty}n^{c_1\mu^{c_2\log(c_3n)}}=\limsup_{n\to\infty}n^{c_1(c_3 n)^{c_2\log \mu}}=\limsup_{n\to\infty}\exp(c_1(c_3 n)^{c_2\log \mu}\log(n))\leq c(c_1,c_2,c_3,\mu)<\infty.
\end{equation*} 

Finally observe that the chosen $\hat k_n$ satisfies $\hat k_n \le \BR2\frac 12 (\log_2(3n) - 11)$ since $\log_2(3n) \ge 21$ \BR2(here we use assumption $n\geq 2^{20}$).\ER 
\end{proof}

For later applications to the heat kernel (see Proposition~\ref{C:boundheatkernel} below) it is useful to replace the $L^2$-norm on the right-hand side in \eqref{est:Tbound} by the $L^1$-norm. This can be achieved by a similar argument as in Step~3 of the proof of Theorem~\ref{T:bound} by replacing in the interpolation inequality \eqref{ineq:esttlocalboundinterpolation} the exponents $\frac1{p'}$ and $1-\frac1{p'}$ with $\frac1{2p'}$ and $1-\frac1{2p'}$, respectively. Since we do not know how to replace the $L^2$-norm on the right-hand side in \eqref{est:Tboundalmostfinal2} by the $L^1$ norm we keep the $L^\infty$ norm of $u$ (to a very small power) on the right-hand side and obtain the following
\begin{corollary}\label{C:Tbound}
 Under the assumptions of Theorem~\ref{T:bound}, there exists $c=c(d,p,q)\in[1,\infty)$ such that the following is true: Let $u>0$ be a \BR2subcaloric function \ER in \BR2$Q_1(t_0,x_0,\tau,n)$ with $t_0\in\R$, $x_0\in\mathbb Z^d$, \ER $n\ge  2^{20}$ and $\tau>0$. Then, 
\begin{equation*}
\|u\|_{L^\infty(Q_\frac12(\BR2t_0,x_0,\ER \tau,n))}\leq c\mathcal C(\omega,\BR2B(x_0,n)\ER,\tau)^\frac{2p}{p-1}\|u\|_{\underline L^{1}(Q_1(\BR2t_0,x_0,\ER\tau,n))}^{1-(\frac12(1+\frac1p))^{\hat k}}\|u\|_{L^{\infty}(Q_1(\BR2t_0,x_0,\ER\tau,n))}^{(\frac12(1+\frac1p))^{\hat k}}
\end{equation*}
%
where $\hat k:=\lfloor\frac14\log_2(3n)\rfloor$.
\end{corollary}

\subsection{Proof of Theorem~\ref{T:oscillation}}

%
With the local boundedness statement Theorem~\ref{T:bound}, the oscillation decay can be proven by already established methods. The following argument is essentially the parabolic version (in the form of \cite[Section~5.2]{WuYinWang}) of Moser's proof, see \cite{Moser60}, of the De Giorgi theorem in the elliptic case. In recent works \cite{AT19,CS19} this strategy is already adapted to the discrete and degenerate situation that we consider here but under more restrictive summability assumption on $\omega$ and $\omega^{-1}$. However, in order to keep the presentation self-contained we provide a detailed proof below.

First, we introduce a suitable regularization of the map $z\mapsto (-\log(z))_+$, defined by
\begin{equation}\label{def:g}
g(z)=\begin{cases}-\log(z)&\mbox{if $z\in(0,\bar c]$},\\\frac{(z-1)^2}{2\bar c(1-\bar c)}&\mbox{if $z\in(\bar c,1]$}\\0&\mbox{if $z\geq1$}\end{cases},
\end{equation}
where $\bar c\in[\frac14,\frac13]$ is the smallest solution of $2c\log(\frac1c)=1-c$. Notice that $g\in C^1((0,\infty))$ is non-negative, convex and non-increasing. 
\begin{lemma}\label{L:weakHarnack}
Fix $d\geq2$ \BR2and $\omega\in\Omega$\ER. Suppose that $u>0$ satisfies $\partial_tu-\mathcal L^\omega u\geq0$ in \BR2$Q(n):=[-n^2,0]\times B(n)$\ER. Fix $\lambda\in(0,1)$ and suppose
\begin{equation}\label{ass:LweakHarnack}
{\rm m}(\{(x,t)\in Q(n),\, u_t(x)\geq1\})\geq \lambda {\rm m}(Q(n))
\end{equation}
(see \eqref{def:m} for the definition of ${\rm m}(\cdot)$). Then, for any 
\begin{equation}\label{ass:sigma12}
\sigma_1\in(0,\lambda) \textrm{ and } \sigma_2\in(\lambda,1) \textrm{ satisfying } \frac{1-\lambda}{1-\sigma_1}\frac{|B(n)|}{|B(\sigma_2 n)|}\le \frac{17}{24}, \quad \textrm{and } \quad n\geq\frac1{1-\sigma_2},
\end{equation}
there exists $h=h(d,\lambda,\|\omega\|_{\underline L^1(B(n))},\sigma_2)\in(0,1)$ such that
\begin{equation}\label{claim:LweakHarnack}
|\{x\in B(\sigma_2 n),\, u_t(x)\geq h\}|\geq \frac14 |B(\sigma_2n)|\quad\mbox{for all }\BR2-\sigma_1 n^2\leq t\leq 0.
\end{equation}
Moreover, there exists $c=c(d)<\infty$ such that \eqref{claim:LweakHarnack} holds with 
\begin{equation}\label{est:minh}
 h=\exp\biggl(-c\biggl(1+\frac{\|\omega\|_{\underline L^1(B(n))}}{(1-\sigma_2)^2\lambda^{d}}\biggr)\biggr).
\end{equation}
\end{lemma}

\begin{proof}[Proof of Lemma~\ref{L:weakHarnack}]

The proof follows the argument of \cite[Lemma~5.2.3]{WuYinWang} (and discrete variants \cite{AT19,CS19}).

\step 1 We claim that there exists $c\in[1,\infty)$ such that for every $\eta:\mathbb Z^d\to[0,1]$ with $ \eta \equiv 0$ in $\mathbb Z^d\setminus B(n-1)$ 
\begin{equation}\label{est:energyg}
 \frac{\mathrm{d}}{\dt}\|\eta^2 g(u_t)\|_{\underline L^1(B(n))}+\frac1{|B(n)|}\sum_{\ee\in\mathbb B^d}\varphi_\eta(\ee)\omega(\ee)(\nabla g(u_t)(\ee))^2\leq c\|\omega\|_{\underline L^1(B(n))}\|\nabla \eta\|_{L^\infty(B(n))}^2{\rm osr}(\eta)^2,
\end{equation}
where ${\rm osr}(\eta):=\max\{\max\{\frac{\eta(y)}{\eta(x)},1\}\,|\, \{x,y\}\in\mathbb B^d,\,\eta(x)\neq0\}$ and $\varphi_\eta(\ee):=\min\{\eta^2(\overline \ee),\eta^2(\underline \ee)\}$ for every $\ee=(\overline \ee,\underline \ee)\in\mathbb B^d$.

\smallskip

Since $\partial_t u-\mathcal L^\omega u\geq0$, we have
\begin{align*}
\partial_t\sum_{x\in\mathbb Z^d}\eta(x)^2g(u_t(x))=&\sum_{x\in\mathbb Z^d}\eta(x)^2g'(u_t(x))\partial_tu_t(x)\\
\leq& \sum_{x\in\mathbb Z^d}\eta(x)^2g'(u_t(x))\mathcal L^\omega u_t(x)\\
=&-\sum_{\ee\in\mathbb B^d}\nabla (\eta^2g'(u_t))(\ee) \omega(\ee)\nabla u_t(\ee).
\end{align*}
Since $\bar c \in [\frac 14,\frac 13]$, we have $\bar c (1-\bar c) \ge \frac{3}{16}$, which gives $\frac13 g'(r)^2\leq g''(r)$ and $-rg'(r)\leq \frac43$ for almost all $r>0$. This combined with Lemma~\ref{L:appendixMartin} then implies 
\begin{equation*}
 -\sum_{\ee\in\mathbb B^d}\nabla (\eta^2g'(u_t))(\ee) \omega(\ee)\nabla u_t(\ee)\leq -\frac16\sum_{\ee\in\mathbb B^d}\varphi_\eta(\ee)\omega(\ee)(\nabla g(u_t)(\ee))^2+6{\rm osr}\,(\eta)^2\sum_{\ee\in\mathbb B^d}\omega(\ee)(\nabla \eta(\ee))^2
\end{equation*}
and thus the claim follows.

\step 2 Conclusion.

Let $\sigma_1,\sigma_2>0$ be such that \eqref{ass:sigma12} is satisfied. For given $h\in(0,1)$ (specified below), we set $w_t(x):=g(u_t(x)+h)$ and
\begin{equation}
\overline \lambda(t):=|\{x\in B(n),\,u_t(x)\geq1\}|,\qquad \mathcal N_t(h):=\{x\in B(\sigma_2 n),\,u_t(x)\geq h\}.
\end{equation} 
Assumption \eqref{ass:LweakHarnack} yields
\begin{equation*}
 \BR2\int_{-n^2}^{0}\ER\overline \lambda(t)\dt\geq\lambda {\rm m}(Q(n))=\lambda n^2|B(n)|,
\end{equation*}
and in combination with the obvious inequality
\begin{equation*}
 \BR2\int_{-\sigma_1 n^2}^{0}\ER\overline \lambda(t)\dt\leq\sigma_1 n^2 |B(n)|,
\end{equation*}
we obtain
\begin{equation*}
 \BR2\int_{-n^2}^{-\sigma_1 n^2}\ER\overline \lambda(t)\dt\geq(\lambda-\sigma_1)n^2 |B(n)|.
\end{equation*}
By the mean value theorem, we find $\tau\in\BR2[-n^2,-\sigma_1n^2]$ such that
\begin{equation}\label{est:lambdatau}
 \overline \lambda(\tau)\geq \frac{\lambda-\sigma_1}{1-\sigma_1}|B(n)|.
\end{equation}
For any $t_2\in\BR2[-\sigma_1n^2,0]$, we deduce from \eqref{est:energyg} (applied to the positive supercaloric function $u+h$) and $\sigma_2\in(\lambda,1)$ together with an 'affine cut-off' $\eta$ satisfying 
\begin{equation}\label{ass:affincutoff}
 \eta\equiv1\;\mbox{in $B(\sigma_2 n)$},\quad\eta\equiv0\mbox{ in $\mathbb Z^d\setminus B(n-1)$,}\quad |\nabla \eta|\lesssim \frac1{n(1-\sigma_2)},\quad {\rm osr}(\eta^2)\leq2
\end{equation}
($n\geq \frac1{1-\sigma_2}$, see \eqref{ass:sigma12}, ensures existence of such $\eta$) that
\begin{equation}\label{est:wt2:1}
\| w_{t_2}\|_{L^1(B(\sigma_2 n))}\leq c\frac{|B(n)|}{(1-\sigma_2)^2}\|\omega\|_{\underline L^1(B(n))}+\|w_\tau\|_{ L^1(B( n))}.
\end{equation}
Since $g\geq0$ is non-increasing, we can estimate the left-hand side in \eqref{est:wt2:1} from below as
\begin{equation}\label{est:wt2:2}
\|w_{t_2}\|_{ L^1(B(\sigma_2 n))}\geq \sum_{x\in B(\sigma_2 n)\setminus \mathcal N_{t_2}}w_{t_2}(x)\geq|B(\sigma_2 n)\setminus \mathcal N_{t_2}(h)|g(2h).\end{equation}
Combining $w_t=0$ on $\{x\in B(n),\,u_t(x)\geq1\}$ (recall $g(z)=0$ for $z\geq1$), the monotonicity of $g$ and \eqref{est:lambdatau}, we obtain
\begin{align}\label{est:wt2:3}
\|w_\tau\|_{L^1(B(n))}\leq(|B(n)|-\overline \lambda(\tau))g(h)\leq \biggl(1-\frac{\lambda-\sigma_1}{1-\sigma_1}\biggr)|B(n)|g(h)=\frac{1-\lambda}{1-\sigma_1}\frac{|B(n)|}{|B(\sigma_2 n)|}|B(\sigma_2 n)|g(h).
\end{align}
Estimates \eqref{est:wt2:1}-\eqref{est:wt2:3} and assumption $\frac{1-\lambda}{1-\sigma_1}\frac{|B(n)|}{|B(\sigma_2 n)|}\le \frac{17}{24}$ yield
\begin{equation}\label{est:LweakHarnack:finalkey}
|B(\sigma_2 n)\setminus \mathcal N_{t_2}(h)|\leq \biggl(\frac{c|B(n)|\|\omega\|_{\underline L^1(B(n))}}{|B(\sigma_2 n)|(1-\sigma_2)^2g(2h)}+\frac {17}{24} \frac{g(h)}{g(2h)}\biggr)|B(\sigma_2n)|.
\end{equation}
Using $g(s)=-\log(s)$ for $s\in(0,\frac14)$ \BR2 and $\frac{17}{24}<\frac34$\ER, we can choose $h>0$ sufficiently small such that
\begin{equation*}
|B(\sigma_2 n)\setminus \mathcal N_{t_2}(h)|\leq \frac34|B(\sigma_2n)|,
\end{equation*}
which by the arbitrariness of $t_2\in\BR2[-\sigma_1n^2,0]$ yield \eqref{claim:LweakHarnack}. The lower estimate on $h$ can easily deduce from \eqref{est:LweakHarnack:finalkey} and the elementary estimate $\frac{B(n)}{B(\sigma_2 n)}\leq\frac{B(n)}{B(\lambda n)}\leq (\BR2\frac3\lambda)^d$.

\end{proof}

Lemma~\ref{L:weakHarnack} and Theorem~\ref{T:bound} yield the following weak Harnack inequality
\begin{theorem}[Weak Harnack inequality]\label{T:weakHarnack}
Fix $d\geq2$, $\omega\in\Omega$ and $p\in(1,\infty]$, $q\in(\frac{d}2,\infty]$ satisfying $\frac1p+\frac1q<\frac2{d-1}$. Let $u>0$ be such that $\partial_t u-\mathcal L^\omega u\geq0$ in $Q(n)\BR2=[-n^2,0]\times B(n)$. Suppose there exist $\e>0$ and $\lambda\in(0,1)$ such that
\begin{equation}\label{assumption:weakharnack}
 {\rm m}(\{(x,t)\in Q(n),\, u_t(x)\geq\e\})\geq \lambda {\rm m}(Q(n)).
\end{equation}
Suppose that $\sigma_1$ and $\sigma_2$ satisfy \eqref{ass:sigma12} and set $\sigma:=\min\{\sqrt\sigma_1,\sigma_2\}$. There exists a constant 
$$
\gamma=\gamma(d,\e,\lambda,p,q,\|\omega\|_{\underline L^p(B(n))},\|\omega^{-1}\|_{\underline L^q(B(n))},\sigma_1,\sigma_2)>0
$$
such that if $\sigma n\ge  2^{20}$, then 
\begin{equation}\label{est:weakharnackclaim}
 u_t(x)\geq\gamma\quad\mbox{for all $(t,x)\in Q_\frac12(\lfloor \sigma n\rfloor)$.}
\end{equation}
Moreover there exists $c=c(d,\lambda,p,q,\sigma_1,\sigma_2)\in[1,\infty)$ such that \eqref{est:weakharnackclaim} holds with
\begin{equation}\label{est:mingamma}
\gamma= \e\exp\left(-c\left(1+\|\omega\|_{\underline L^1(B(n))}+\mathcal C(\omega,B(n))^{\frac{2p}{p-1}}\|\omega^{-1}\|_{\underline L^\frac{d}2(B(n))}\right)\right),
\end{equation}
where $\mathcal C(\omega,B(n)):=\mathcal C(\omega,B(n),1)$ and $\mathcal C(\omega,B(n),1)$ is defined in \eqref{def:C}.
\end{theorem}

\begin{proof}[Proof of Theorem~\ref{T:weakHarnack}]

Without loss of generality, we assume $\e=1$.

Consider the function $(t,x)\mapsto W_t(x):=G(u_t(x))$, where $G(s):= g(\frac{s+\gamma }h)$ with $s\in\mathbb R$ and \BR2 suitable \ER constants $0<\gamma <h$ which are specified later. 

\step 1 $W$ is a \BR2subcaloric function\ER, i.e.\ $\frac{\mathrm{d}}{\dt} W_t(x)-\mathcal L^\omega W_t(x)\leq0$ for all $(t,x)\in Q(n)$. Indeed, this is a consequence of the convexity of $G$ in the form
\begin{align*}
 \mathcal L^\omega W_t(x)&=\sum_{y\sim x}\omega(x,y)(G(u_t(y))-G(u_t(x)))\\
 &\ge G'(u_t(x))\sum_{y\sim x}\omega(x,y)((u_t(y))-u_t(x))=G'(u_t(x))\mathcal L^\omega u_t(x),
\end{align*}
combined with the fact that $u$ is supercaloric and $G'\leq0$, and thus
\begin{align*}
\partial_t W_t(x)-\mathcal L^\omega W_t(x)\leq G'(u_t(x))(\partial_t u_t(x)-\mathcal L^\omega u_t(x))\leq0.
\end{align*}
\step 2 Let $\sigma_1\in(0,\lambda)$ and $\sigma_2\in(\lambda,1)$ be such that $\frac{1-\lambda}{1-\sigma_1}\frac{|B(n)|}{|B(\sigma_2 n)|}\le  \frac{17}{24}$  with $\sigma_2 \le 1 - \frac 1n$ and let $h=h(d,\lambda,\|\omega\|_{\underline L^1(B(n))},\sigma_2)\in(0,1)$ be as in Lemma~\ref{L:weakHarnack}. We claim that there exists $c=c(d)\in[1,\infty)$ such that
\begin{align}\label{est:weakHarnackpoincare}
\BR2\fint_{-\sigma_1 n^2}^{0}\ER\|W_t\|_{\underline L^2(B(\sigma_2 n))}^2\dt\leq c\frac{\sigma_2^2}{\sigma_1}\|\omega^{-1}\|_{\underline L^{\frac{d}2}(B(\sigma_2 n))}\lambda^{-d}\biggl(\frac{\|\omega\|_{\underline L^1(B(n))}}{(1-\sigma_2)^{2}}+g(\tfrac{\gamma }{h})\biggr).
\end{align}
Computation analogous to the one leading to~\eqref{est:energyg} in Step 1 of the proof of Lemma~\ref{L:weakHarnack}, leads to the following: \ER for any $\eta:\mathbb Z^d\to[0,1]$ with $\eta=0$ in $\mathbb Z^d\setminus B(n-1)$
\begin{equation}\label{est:energywH1}
 \frac{\mathrm d}{\dt}\|\eta^2 W_t\|_{\underline L^1(B(n))}+|B(n)|^{-1}\sum_{\ee\in\mathbb B^d}\varphi_{\eta}(\ee)\omega(\ee)(\nabla W_t(\ee))^2\leq c\|\omega\|_{\underline L^1(B(n))}\|\nabla \eta\|_{L^\infty}^2({\rm osr}\,\eta)^2.
\end{equation}
Choosing a suitable cut-off function satisfying $\eta=1$ in $B(\sigma_2n)$ (hence $\varphi_{\eta}(\ee) = 1$ in $B(\sigma_2n)$), $\|\nabla \eta\|_{L^\infty}\lesssim (n(1-\sigma_2))^{-1}$ and $({\rm osr}(\eta))\leq2$, we deduce from \eqref{est:energywH1} and the monotonicity of $g$ that
\begin{align}\label{est:weakHarnacknablaw}
\int_{t_0-\sigma_1n^2}^{t_0}\|\sqrt{\omega} \nabla W_t\|_{\underline L^2(B(\sigma_2n))}^2 &\le \frac{|B(n)|}{|B(\sigma_2 n)|}\biggl(c\frac{\sigma_1n^2\|\omega\|_{\underline L^1(B(n))}}{n^2(1-\sigma_2)^{2}}+\|W_{t_0-\sigma_1 n^2}\|_{\underline L^1(B(n))}\biggr)\notag
\\
&\le\frac{|B(n)|}{|B(\sigma_2 n)|}\biggl(c\frac{\|\omega\|_{\underline L^1(B(n))}}{(1-\sigma_2)^{2}}+g(\tfrac{\gamma}{h})\biggr).
\end{align}
Assumption \eqref{assumption:weakharnack} (recall that we suppose $\e=1$) together with Lemma~\ref{L:weakHarnack} implies
\begin{equation}\label{est:weakHarnacksmallset}
|\{x\in B(\sigma_2 n),\, u_t(x)\geq h\}|\geq \frac14|B(\sigma_2n)|\quad\mbox{for all }t_0-\sigma_1 n^2\leq t\leq t_0,
\end{equation}
with $h$ satisfying~\eqref{est:minh}. Estimate \eqref{est:weakHarnackpoincare} follows from \eqref{est:weakHarnacknablaw} and \eqref{est:weakHarnacksmallset} together with the following Poincar\'e-type inequality: There exists $c=c(d)\in[1,\infty)$ such that for all $\emptyset\neq \mathcal N\subset B:=B(m)$ and $v:B\to\R$ it holds
\begin{align}\label{ineq:weightedsobharnack}
\|v-(v)_{\mathcal N}\|_{\underline L^2(B)}\leq& c\biggl(1+\frac{|B|}{|\mathcal N|}\biggr)|B|^\frac1d \|\omega^{-1}\|_{\underline L^\frac{d}2(B)}^\frac12\|\sqrt{\omega}\nabla v\|_{\underline L^2(B)},
\end{align}
where we recall $(f)_{\mathcal N} :=\ER \frac{1}{|\mathcal{N}|} \sum_{x \in \mathcal N} f(x)$. Before recalling the argument for \eqref{ineq:weightedsobharnack} we discuss how it is used to deduce \eqref{est:weakHarnackpoincare}. By definition, we have $W_t(x)=0$ for all $x\in \mathcal N_t:=\{x\in B(\sigma_2 n),\, u_t(x)\geq h\}$ and thus for all $t\in[t_0-\sigma_1n^2,t_0]$
\begin{align*}
\|W_t\|_{\underline L^2(B(\sigma_2 n))}=\|W_t-(W_t)_{\mathcal N_t}\|_{\underline L^2(B(\sigma_2 n))}\leq 5c|B(\sigma_2n)|^\frac1d \|\omega^{-1}\|_{\underline L^\frac{d}2(B(\sigma_2 n))}^{\frac 12}\|\sqrt{\omega}\nabla W_t\|_{\underline L^2(B(\sigma_2n)))}.
\end{align*}
Squaring the above expression and integrating in time from \BR2$-\sigma_1n^2$ to $0$\ER, we obtain \eqref{est:weakHarnackpoincare} using \eqref{est:weakHarnacknablaw}. 

\smallskip

Finally, we recall the computations that yield \eqref{ineq:weightedsobharnack}: For every $s\in[1,d)$, we have (with the notation of Proposition~\ref{T:sob})
\begin{align}\label{est:sobolev:fraction}
\|v-(v)_{\mathcal N}\|_{\underline L^{s_d^*}(B)}&\uoset{\eqref{est:sobolev:bulk}}{}\le \|v-(v)_{B}\|_{\underline L^{s_d^*}(B)}+|(v)_{B}-(v)_{\mathcal N}|\notag
\\
&\uoset{\eqref{est:sobolev:bulk}}{}{\le} \biggl(1+\frac{|B|}{|\mathcal N|}\biggr)\|v-(v)_{B}\|_{\underline L^{s_d^*}(B)}\notag
\\
&\stackrel{\eqref{est:sobolev:bulk}}{\le} c(d,s)\biggl(1+\frac{|B|}{|\mathcal N|}\biggr)|B|^\frac1d \|\nabla v\|_{\underline L^s(B)}.
\end{align}
Estimate \eqref{ineq:weightedsobharnack} follows from \eqref{est:sobolev:fraction} with $s=\frac{2d}{d+2}$ (and thus $s_d^*=2$) and H\"older inequality in the form
$$
\|\nabla v\|_{\underline L^\frac{2d}{d+2}(B)}\leq  \|\omega^{-1}\|_{\underline L^\frac{d}2(B)}^\frac12\|\sqrt{\omega}\nabla v\|_{\underline L^2(B)}.
$$
%

\step 3 Conclusion. 

Using that $W$ is a \BR2subcaloric function \ER in $\BR2[-\sigma n^2,0]\times B( \sigma n)$ with $\sigma=\min\{\sqrt{\sigma_1},\sigma_2\}$, we obtain from Theorem~\ref{T:bound} and \eqref{est:weakHarnackpoincare} (combined with the assumption $\sigma n\ge 2^{20}$)
\begin{align*}
\sup_{(t,x)\in Q_{\frac12}(\lfloor \sigma n\rfloor)}W_t(x)^2&\stackrel{\eqref{est:Tbound}}{\lesssim} \mathcal C(\omega,B(\sigma n))^{\frac{2p}{p-1}}\|W\|^{2}_{\underline L^2(Q_1(\lfloor \sigma n\rfloor))}
\\
&\stackrel{\eqref{est:weakHarnackpoincare}}{\le} c\mathcal C(\omega,B(n))^{\frac{2p}{p-1}}\lambda^{-d}\|\omega^{-1}\|_{\underline L^{\frac{d}2}(B(n))}\biggl(\frac{\|\omega\|_{\underline L^1(B(n))}}{(1-\sigma_2)^2}+ g(\tfrac{\gamma}{h})\biggr),
\end{align*}
where $c=c(d,p,q,\sigma_1,\sigma_2)\in[1,\infty)$ and $\mathcal C(\omega,B(n)):=\mathcal C(\omega,B(n),1)$ (see \eqref{def:C}). Choose $\gamma\in(0,\frac{h}2)$ sufficiently small such that
\begin{equation}\label{ass:weakharnackcontra}
\left(g\left(\frac{2\gamma}h\right)\right)^2>c\mathcal C(\omega,B(n))^{\frac{2p}{p-1}}\lambda^{-d}\|\omega^{-1}\|_{\underline L^\frac{d}2(B(n))}\biggl(\frac{\|\omega\|_{\underline L^1(B(n))}}{(1-\sigma_2)^2}+ g(\tfrac{\gamma}{h})\biggr).
\end{equation}
%
Then it holds $u\ge\gamma$ on $Q_\frac12(\lfloor\sigma n\rfloor)$. Indeed, if there were $(\tilde t,\tilde x)\in Q_\frac12(\lfloor\sigma n\rfloor)$ with $u_{\tilde t}(\tilde x)<\gamma$, by  monotonicity of $g$
\begin{equation*}
 \left(g\left(\frac{2\gamma}h\right)\right)^2\leq \left(g\left(\frac{u_{\tilde t}(\tilde x)+\gamma}h\right)\right)^2\leq c\mathcal C(\omega,B(n))^{\frac{2p}{p-1}}\lambda^{-d}\|\omega^{-1}\|_{\underline L^\frac{d}2(B(n))}\biggl(\frac{\|\omega\|_{\underline L^1(B(n))}}{(1-\sigma_2)^2}+ g(\tfrac{\gamma}{h})\biggr)
\end{equation*}
which contradicts \eqref{ass:weakharnackcontra}. Finally, we notice that \eqref{ass:weakharnackcontra} is satisfied for any $\gamma$ with
%
\begin{equation}\label{est:mingammapf}
 \gamma<h\exp\biggl(-\tfrac12{\rm A}-\log(2)-\sqrt{({\tfrac12\rm A}+\log(2))^2+{\rm A}\frac{\|\omega\|_{\underline L^1(B(n))}}{(1-\sigma_2)^2}-\log(2)^2}\biggr).
\end{equation}
where
$$
{\rm A}:=c\mathcal C(\omega,B(n))^{\frac{2p}{p-1}}\lambda^{-d}\|\omega^{-1}\|_{\underline L^{\frac{d}2}(B(n))},
$$
and $c=c(d,p,q,\sigma_1,\sigma_2)\in[1,\infty)$. Combining \eqref{est:mingammapf} with \eqref{est:minh}, we obtain \eqref{est:mingamma}.

\end{proof}

Theorem~\ref{T:oscillation} follows from the weak Harnack inequality Theorem~\ref{T:weakHarnack} using classical arguments adapted to the discrete setting:

\begin{proof}[Proof of Theorem~\ref{T:oscillation}]
Without loss of generality, we consider $t_0=0$ and $x_0=0$ and suppose $p,q<\infty$ (the case $p=q=\infty$ is classical; if, for instance, $p<\infty$ and $q=\infty$, we use the statement with $p$ and $\tilde q\in(\frac{d}2,\infty)$ satisfying $\frac1p+\frac1{\tilde q}<\frac{2}{d-1}$ combined with the trivial inequality $\|\omega^{-1}\|_{\underline L^{\tilde q}(B)}\leq \|\omega^{-1}\|_{\underline L^q(B)}$). Modifying $u$ by a constant if necessary, without loss of generality we can assume that
\begin{equation*}
M:=\max_{Q_1(n)}u=-\min_{Q_1(n)}u=\frac12 {\rm osc}_{Q_1(n)}u.
\end{equation*}
Moreover, w.l.o.g., we \BR2 assume ${\rm m}(\{(x,t)\in Q_1(n),\, u\geq0\})\geq \frac12{\rm m}(Q_1(n))$, since otherwise we consider $-u$ instead\ER. Thus, the function $v=1+\frac{u}M$ satisfies $\partial_t v-\mathcal L^\omega v=0$ and
\begin{equation*}
 {\rm m}(\{(x,t)\in Q_1(n),\, v\geq1\})\geq \frac12{\rm m}(Q_1(n)).
\end{equation*} 
We choose $\sigma_1 = \frac14$ and $\sigma_2^d = \frac{49}{51}$, and observe that this choice satisfies 
\begin{equation}\label{T:osc:sigma}
\text{(i)}:\quad\frac{\BR2\frac12}{1-\sigma_1}\frac{|B(n)|}{|B(\sigma_2 n)|} \le \frac{17}{24}\quad\mbox{and}\quad\text{(ii):}\quad n \ge \frac{1}{1-\sigma_2}\qquad\mbox{\BR2for all $n\geq 50d$}.
\end{equation}
\BR2Before we give the elementary argument for \eqref{T:osc:sigma} we show that \eqref{T:osc:sigma} and Theorem~\ref{T:weakHarnack} imply the desired claim with $N:=\max\{2^{21},50d\}$: For $n\geq N$, Theorem~\ref{T:weakHarnack} with $\lambda=\frac12$ and $\sigma_1,\sigma_2$ as above (and thus $\sigma n\geq \frac12 N=2^{20}$) yields
\begin{equation}\label{est:vgammosc}
 v_t(x)\geq \gamma\qquad \mbox{for all $(t,x)\in Q_\frac18(n)$,}
\end{equation}
for some $\gamma=\gamma(d,p,q,\|\omega\|_{\underline L^p(B(n))},\|\omega^{-1}\|_{\underline L^q(B(n))})>0$.
Hence,
\begin{equation*}
-M(1-\gamma)\leq u_t(x)\leq M\qquad\mbox{for all $(t,x)\in Q_\frac18(n)$}
\end{equation*}
which implies ${\rm osc}_{Q_\frac18}u\leq 2M(1-\frac{\gamma}2)=\theta{\rm osc}_{Q_1}u$ with $\theta:=1-\frac{\gamma}2\in(0,1)$, which concludes the argument.

Finally we give the argument for \eqref{T:osc:sigma}: Inequality (ii) in \eqref{T:osc:sigma} follows from concavity of $t \mapsto t^{\frac1d}$ in the form $1-\sigma_2 = 1 - \bigl( \frac{49}{51} \bigr)^{\frac 1d} \ge (1-\frac{49}{51})\frac{1}{d}1^{\frac1d - 1} = \frac{2}{51d}$ and $\frac{51}2\leq 50$. Inequality (i) in \eqref{T:osc:sigma} can be written as $48 |B(n)| \le 51 |B(\sigma_2n)|$. Since $|B(\sigma_2n)|=(2\lfloor \sigma_2 n \rfloor+1)^d\geq (2\sigma_2 n-1)^d$ it suffices to show that  $\bigl(\frac{51}{48}\bigr)^{\frac1d} (2\sigma_2 n -1) \ge 2n+1$, which in turn is equivalent to $n\geq \frac12\frac{1+(\frac{51}{48})^\frac1d}{(\frac{49}{48})^\frac1d-1}$. Using $1+(\frac{51}{48})^\frac1d\leq 1+\frac{51}{48}$ and concavity of $t \mapsto t^{\frac 1d}$ in the form $(\frac{49}{48})^\frac1d-1\geq \frac1d(\frac{49}{48}-1)=\frac1{48d}$, we obtain that (i) in \eqref{T:osc:sigma} is satisfied for $n\geq \frac12 48d(1+\frac{51}{48})=d\frac{99}2$.

\end{proof}
\ER

\section{Local limit theorem}

\subsection{Some properties of the heat kernel}

In this section, we use the local boundedness result Theorem~\ref{T:bound} to derive a deterministic on-diagonal upper bounds on the heat kernel (see Proposition~\ref{C:boundheatkernel}). This upper bound combined with Theorem~\ref{T:oscillation} implies large-scale H\"older-continuity of the heat kernel (see Proposition~\ref{P:holderheat}) which will be a crucial ingredient in the proof of the local limit theorem. As a side result, we obtain an on-diagonal heat kernel estimate, see Corollary~\ref{P:ondiagonalbound}.

\medskip

Next we apply the local boundedness for subcaloric functions (in the form of Corollary~\ref{C:Tbound}) to the heat kernel $p^\omega$ of $X$. For this, we recall that for fixed $x\in\mathbb Z^d$ the map $[0,\infty)\times \mathbb Z^d\ni(t,y)\mapsto p_t(x,y)$ solves the Cauchy problem 
\begin{equation}\label{def:pt2}
 \partial_t p^\omega(x,\cdot)-\mathcal L^\omega p(x,\cdot)=0\quad\mbox{on $(0,\infty)\times \mathbb Z^d$ and}\quad p_0(x,y)=\delta_x(y),
\end{equation}
where $\delta_x(x)=1$ and $\delta_x(z)=0$ for $z\neq x$. 
\begin{proposition}\label{C:boundheatkernel}
 Fix $d\ge 2$, $\omega\in\Omega$, and $p\in(1,\BR2\infty)$, $q\in(\frac{d}2,\BR2\infty)$ satisfying $\frac1p+\frac1q<\frac2{d-1}$. There exists $c=c(d,p,q)\in[1,\infty)$ such that for every $x,y\in\mathbb Z^d$ 
\begin{equation}\label{est:C:heatkernel}
  p_t^\omega(x,y)\leq c  \mathcal C\biggl(\omega,B(y,\sqrt{t})\ER\biggr)^{\frac{2p}{p-1}} t^{-\frac{d}2}\qquad\mbox{for all $t\geq1$,}
\end{equation}
 where $\mathcal C(\omega,B):=\mathcal C(\omega,B,1)$ and $\mathcal C(\omega,B,1)$ is defined in \eqref{def:C}.
\end{proposition}

\begin{remark}\label{rem:trapping}
\BR2Well-known examples of trapping of random walks, see e.g.\ \cite{BKM15}, show that the statement of Proposition~\ref{C:boundheatkernel} fails for $q<\frac{d}2$: Fix $q<\frac{d}2$. For $n \in\mathbb N$ choose $\omega\in\Omega$ as
$$
\omega(x,y)=\begin{cases}n^{-\frac{d}q}&\mbox{if $x=0$ and $|y|=1$}\\1&\mbox{otherwise}\end{cases}.
$$
Obviously, we have $\omega\leq1$ and  $\|\omega^{-1}\|_{\underline L^q(B(n))}\leq c(d)<\infty$. Moreover, an elementary computation yields
\begin{equation}\label{est:heatfail}
p_t^\omega(0,0)\geq1-tc(d) n^{-\frac{d}q}\qquad\mbox{for all $t>0$.}
\end{equation}
Clearly, \eqref{est:heatfail} with $q<\frac{d}2$ (and thus $2-\frac{d}q<0$) contradicts the validity of an estimate of the form \eqref{est:C:heatkernel} for $x=y=0$ and $t=n^2$, i.e. $p_{n^2}^\omega(0,0)\leq c(d,p,q)n^{-d}$, for $n$ sufficiently large. Since Proposition~\ref{C:boundheatkernel} follows directly from local boundedness in the form of Corollary~\ref{C:Tbound} the above argument shows that assumption $q>\frac{d}2$ is essential in Corollary~\ref{C:Tbound}.

\end{remark}

\begin{proof}[Proof of Proposition~\ref{C:boundheatkernel}]
By translation it suffices to prove the claim for $y=0$, and we use the shorthand $p_t^\omega=p_t^\omega(\BR2x,\cdot)$.

\smallskip

For \BR2$n\in\mathbb N$ specified \ER below, we set 
\begin{equation}\label{def:Q12Q1}
 Q_1:=[0,2 n^2]\times B(n),\quad Q_\frac12 := [n^2,2n^2]\times B(\tfrac{n}2).
\end{equation}
Without loss of generality, we suppose from now on that
\begin{equation}
\sqrt{t}\ge 2^{20} \textrm{ and choose } n:=\lfloor \sqrt{t}\rfloor
\end{equation}
(for $t\in[1,2^{40}]$ estimate \eqref{est:C:heatkernel} is valid with $c=2^{20d}$ since $p^\omega\leq1$). We claim that there exists $c=c(d,p,q)\in[1,\infty)$ such that
\begin{equation}\label{est:boundpsf}
 \sup_{(s,z)\in Q_\frac12} p_s^\omega(z)\leq c \mathcal C(\omega,B(n))^{\frac{2p}{p-1}}n^{-d}.
\end{equation}
%
Indeed, a direct consequence of Corollary~\ref{C:Tbound} (with $\tau=2$) combined with the fact
$$
\|p_t^\omega\|_{L^\infty(\mathbb Z^d)}\leq\|p_t^\omega\|_{L^1(\mathbb Z^d)}=1\qquad\textrm{ for all }t\ge0,
$$
is the existence of $c=c(d,p,q)\in[1,\infty)$ such that
\begin{equation*}
 \sup_{(s,z)\in Q_\frac12} p_s^\omega(z)\le c \mathcal C(\omega,B(n),2)^{\frac{2p}{p-1}} n^{-d(1-\frac12(1+\frac1p)^{\hat k})}, 
\end{equation*}
where $\hat k:=\lfloor\frac14\log_2(3n)\rfloor$ and the $n$-factor comes from averaging of $L^1$ over $Q_1$. The claimed estimate \eqref{est:boundpsf} follows from $\mathcal C(\omega,B,2) \le \BR2\sqrt{2}\ER \mathcal C(\omega,B,1)$ and $p>1$, thanks to which $\frac12 (1+\frac 1p) < 1$ and hence
\begin{equation*}
 \limsup_{n\to\infty}n^{d(\frac12(1+\frac1p))^{\frac14\log_2(3n)}}\leq c(d,p)<\infty.
\end{equation*}
From $n = \lfloor \sqrt{t} \rfloor \le \sqrt{t}$ it follows $n^2 \le t$ while the lower bound on $t$ gives $t \le 2n^2$, hence $(t,0) \in Q_{\frac 12}$, and~\eqref{est:boundpsf} yields
\begin{equation*}
  p_t^\omega(0) \le c \mathcal C(\omega,B(n))^{\frac{2p}{p-1}}n^{-d} \le c 2^{\frac d2}\mathcal C(2\omega,B(n))^{\frac{2p}{p-1}}t^{-\frac d2},
\end{equation*}
which concludes the proof. 
\end{proof}

Next, we combine Proposition~\ref{C:boundheatkernel} with Theorem~\ref{T:oscillation} to obtain large-scale (H\"older-) continuity of the heat kernel provided, we control $\|\omega\|_{\underline L^p(B)}$ and $\|\omega^{-1}\|_{\underline L^q(B)}$ in the limit $|B|\to\infty$. 
\begin{proposition}\label{P:holderheat}
Fix $d\geq2$, $\delta\in(0,1]$, $M\in[1,\infty)$, 
$\omega\in\Omega$ and $p\in(1,\infty]$, $q\in(\frac{d}2,\infty]$ satisfying $\frac1p+\frac1q<\frac2{d-1}$. 
There exist $c=c(d,M,p,q)\in[1,\infty)$  and $\rho=\rho(d,M,p,q)\in(0,1)$ such that the following is true: Suppose that for every $x \in \R^d$ 
\begin{equation}\label{ass:limomega}
 \limsup_{n\to\infty}\left(\|\omega\|_{\underline L^p(B(\lfloor nx\rfloor,n))}+ \|\omega^{-1}\|_{\underline L^q(B(\lfloor nx\rfloor,n))}\right)\leq M.
 \end{equation}
Then for $ t\ge 4\delta^2$
%
\begin{equation}\label{est:heatkernel:holder}
 \limsup_{n\to\infty}\max_{y_1,y_2\in B(\lfloor nx\rfloor,\delta n)\atop s_1,s_2\in[t-\delta^2,t]}n^d|p_{n^2s_1}^\omega(0,y_1)-p_{n^2s_2}^\omega(0,y_2)|\leq c\biggl(\frac{\delta}{\sqrt{t}}\biggr)^\rho t^{-\frac{d}2}.
\end{equation}
\end{proposition}

\begin{proof}[Proof of Proposition~\ref{P:holderheat}]
\BR2Without loss of generality we assume $p,q<\infty$. Let $x\in \R^d$ be fixed. \ER For $k\in\BR2 \mathbb N_0 \cup \{-1\}$, we set $\delta_k:=\frac12 8^{-k}\sqrt t$ and consider the sequence of parabolic cylinders
\begin{equation*}
 Q_k:=n^2[t-\delta_k^2,t]\times B(\lfloor nx\rfloor,\delta_k n).
\end{equation*}
Set $k_0:=\max\{k\in\mathbb N,\,\delta_k\geq \delta\}$. Since $k_0$ is finite and $B(\lfloor nx\rfloor,\delta_kn) = B(\lfloor \delta_kn \frac{x}{\delta_k}\rfloor,\delta_kn)$, using 
assumption~\eqref{ass:limomega} for finitely many points $\frac{x}{\delta_k}$, $k = -1, \ldots, k_0$, we see that 
\begin{equation}\label{omega:control}
\|\omega\|_{\underline L^p(B(\lfloor nx\rfloor,\delta_k n))} + \|\omega^{-1}\|_{\underline L^q(B(\lfloor nx\rfloor,\delta_k n))} \leq M, \quad k = -1,\ldots,k_0 ,
\end{equation}
provided $n$ is sufficiently large. 
Combining this with Theorem~\ref{T:oscillation} 
\ER we find $\overline \theta=\overline \theta(d,M,p,q)\in(0,1)$ such that for $n$ sufficiently large it holds
\begin{equation*}
 \osc_{Q_k}p^\omega\leq \overline\theta  \osc_{Q_{k-1}}p^\omega\qquad \textrm{for all } 1 \le k\leq k_0,
\end{equation*}
where we use the shorthand $p^\omega=p_{\cdot}^\omega(0,\cdot)$, and thus (by iteration)
\begin{equation}\label{est:heatkernel:holder:iteration}
 {\rm osc}_{Q_{k_0}}p^\omega\leq {\overline\theta}^{k_0}  {\rm osc}_{Q_{0}}p^\omega.
\end{equation}
The claimed estimate \eqref{est:heatkernel:holder} is a consequence of \eqref{est:heatkernel:holder:iteration} combined with the following three facts
\begin{align*}
 &n^2[t-\delta^2,t]\times B(\lfloor nx\rfloor,\delta n)\subset Q_{k_0},\\
 &{\overline \theta}^{k_0}\leq {\overline \theta}^{-1}\biggl( \frac{2\delta}{\sqrt{t}}\biggr)^{\frac13|\ln(\overline\theta)|} ,\\
 & {\rm osc}_{Q_{0}}p^\omega\leq 2\|p^\omega\|_{L^\infty(n^2t[\frac34 ,1]\times B(\lfloor nx\rfloor, \frac{\sqrt t}{2} \ER n)}\stackrel{\eqref{est:C:heatkernel}}{\leq} c n^{-d} t^{-\frac{d}2},
\end{align*}
where $c=c(d,M,p,\BR2q)\in[1,\infty)$. In order to apply~\eqref{est:C:heatkernel}, we used that $\mathcal C(\omega,B(y,\sqrt{t}))$ appearing in~\eqref{est:C:heatkernel} is up to a  multiplicative constant controlled by $\mathcal C(\omega,B(\lfloor nx \rfloor, 4n \sqrt{t})) = \mathcal C (\omega,B(\lfloor nx \rfloor, \delta_{-1} n)) \stackrel{\eqref{omega:control}}{\le} c(d,M,p,q)$. 
\end{proof}

From Proposition~\ref{C:boundheatkernel}, we directly deduce the corresponding heat kernel estimate for the random conductance model. This extends \cite[Proposition~3.6]{MO16} to the case of unbounded conductances.
\begin{corollary}\label{P:ondiagonalbound}
Suppose that Assumption~\ref{ass} is satisfied and that there exists $p\in(1,\BR2\infty)$, $q\in(\frac{d}2,\BR2\infty)$ satisfying $\frac1p+\frac1q<\frac2{d-1}$ such that \eqref{ass:moment} is valid. Let $\nu=\nu(d,p,q)\in(0,1)$ be as in \eqref{def:nu}. Then there exists a random variable $\mathcal X\geq0$ such that
\begin{equation}\label{ondiagonalbound:moments}
  \mathbb E[\mathcal X^r]<\infty \quad \forall 0<r<\tfrac{p-1}{p}(1-\nu)(\tfrac1q+\tfrac{2-\nu}p)^{-1}
\end{equation}
and
\begin{equation}\label{est:ondiagonal}
 \sup_{|x|\leq\sqrt{t}} p_t^\omega(0,x)\leq t^{-\frac{d}2}\mathcal X(\omega)\quad\mbox{for all $t\geq1$}.
\end{equation}
\end{corollary}

\begin{proof}[Proof of Corollary~\ref{P:ondiagonalbound}]

By Proposition~\ref{C:boundheatkernel} and the definition of $\mathcal C$ (see \eqref{def:C}) there exists $c=c(d,p,q)\in[1,\infty)$ such that for every $t\geq1$
\begin{equation*}
\sup_{|y|\leq \sqrt{t}}p_t^\omega(0,y)\leq c  \mathcal C\bigl(\omega,B(x, 2 \sqrt{t})\bigr)^{\frac{2p}{p-1}} t^{-\frac{d}2}\leq c\mathcal C_{\rm max}(\omega)t^{-\frac{d}2},
\end{equation*}
with $\mathcal C_{\rm max}$ given by
\begin{equation*}
\mathcal C_{{\rm max}}(\omega):=\max\{1,\bigl(\mathcal M_q(\omega^{-1})(1+ \mathcal M_p(\omega))^{2-\nu}\bigr)^{\frac1{1-\nu}\frac{p}{p-1}}\}\in[1,\infty),
\end{equation*}
where $\mathcal M$ denotes the maximal operator given by
\begin{equation*}
(\mathcal M_r(f))^r:=\sup_{n\in\mathbb N}\|f\|_{\underline L^r(B(n))}^r\qquad\mbox{for all $r\geq1$ and $f:\mathbb Z^d\to\R$}.
\end{equation*}
Hence, we have \eqref{est:ondiagonal} with $\mathcal X=c(\mathcal C_{{\rm max}}(\omega))$ and the claimed moment conditions \eqref{ondiagonalbound:moments} easily follow from H\"older inequality and the $L^r$ ($1<r\leq\infty$) inequalities for the maximal operator (see e.g.\ \cite[Corollary A.2]{MO16}).
\end{proof}

\subsection{Proof of Theorem~\ref{T:Locallimit}}

By now it is well-established that quenched invariance principles (see Theorem~\ref{T}) combined with additional regularity properties of the heat kernel yield local limit theorems, see \cite{ADS16,BH09}. Hence we only provide sketch of the proof.
 
\begin{proof}[Proof of Theorem~\ref{T:Locallimit}]
\BR2We only show that for every $x\in\R^d$ and $t>0$ 
\begin{equation}\label{eq:claimlocallimitpointwise}
 \lim_{n\to\infty}|n^dp_{n^2t}^\omega(0,\lfloor nx\rfloor)-k_t(x)|=0\qquad\mathbb P\mbox{-a.s.},
\end{equation}
\ER
where the Gaussian heat kernel $k_t$ is defined in~\eqref{def:kt}. From the pointwise result \eqref{eq:claimlocallimitpointwise} the desired claim \eqref{eq:claimlocallimit} follows by a covering argument exactly as in \cite[proof of Proposition~3.1]{CS19}.

For given $x\in\mathbb R^d$ and $\delta>0$, we introduce
\begin{equation*}
 \square(x,\delta):=x+[-\delta,\delta]^d,\quad \hat\square_n (x,\delta):=(\square(nx,n\delta))\cap\mathbb Z^d,
\end{equation*}
and recall the elementary fact
\begin{equation}\label{lim:discrete}
\lim_{n\to\infty}\frac{n^d(2\delta)^d}{|\hat\square_n(x,\delta)|}=1.
\end{equation}
We write
\begin{equation*}
 J_n(t,x):=n^dp_{n^2t}^\omega(0,\lfloor nx\rfloor)-k_t(x)=\sum_{i=1}^4J_{i,n}(t,x,\delta),
\end{equation*}
where
\begin{align*}
 J_{1,n}(t,x,\delta):=&\frac{n^d}{|\hat\square_n(x,\delta)|}\sum_{z\in\hat \square_n(x,\delta)}p_{n^2t}^\omega(0,\lfloor nx\rfloor)-p_{n^2t}^\omega(0,z),\\
 J_{2,n}(t,x,\delta):=&\frac{n^d}{|\hat\square_n(x,\delta)|}\biggl(\sum_{z\in\hat \square_n(x,\delta)}p_{n^2t}^\omega(0,z)-\int_{\square(x,\delta)}k_t(y)\,dy\biggr)\\
 J_{3,n}(t,x,\delta):=&\frac{n^d}{|\hat\square_n(x,\delta)|}\biggl(\int_{\square(x,\delta)}(k_t(y)-k_t(x))\,dy\biggr)\\
 J_{4,n}(t,x,\delta):=&k_t(x)\biggl(\frac{n^d(2\delta)^d}{|\hat\square_n(x,\delta)|}-1\biggr).
\end{align*}
It suffices to show
\begin{equation}\label{lim:locallim4}
 \limsup_{\delta\to0}\limsup_{n\to\infty}|J_{i,n}(t,x,\delta)|=0\quad\mbox{for all $i=\{1,2,3,4\}$}\quad \BR2 \mathbb P\mbox{-a.s.}.
\end{equation}
\ER A combination of \eqref{lim:discrete} and the local Lipschitz-continuity of the heat kernel $k$ yield \eqref{lim:locallim4} for $i=3$ and $i=4$. For $i=2$, the convergence in \eqref{lim:locallim4} follows directly form the quenched invariance principle Theorem~\ref{T} and finally for $i=1$, we note that by Proposition~\ref{P:holderheat} and Lemma~\ref{L:ergodic} below
\begin{equation*}
 \limsup_{n\to\infty}|J_{1,n}(t,x,\delta)|\leq c\biggl(\frac{\delta}{\sqrt{t}}\biggr)^\rho t^{-\frac{d}2},\quad \mathbb P\mbox{-a.s.}.
\end{equation*}  
where the right-hand side tends to zero as $\delta\to0$.

\end{proof}

In the proof of Theorem~\ref{T:Locallimit} we used the following consequence of the spatial ergodic theorem:
\begin{lemma}\label{L:ergodic}
Suppose that Assumption~\ref{ass} and \ref{ass2} are satisfied. Then, there exists $c=c(d)\in[1,\infty)$ such that 
\begin{equation}\nonumber
\sup_{x\in\R^d}\limsup_{n\to\infty}(\|\omega\|_{\underline L^p(B(\lfloor nx\rfloor,n))}^p+\|\omega^{-1}\|_{\underline L^q(B(\lfloor nx\rfloor,n))}^q)\leq c\sum_{x\in\mathbb Z^d}(\mathbb E[\omega(0,x)^p]+\mathbb E[\omega(0,x)^{-q}])<\infty\quad \BR2 \mathbb P\mbox{-a.s.}.
\end{equation} 
\end{lemma}

\section{Elliptic regularity: Proof of Theorem~\ref{P:oscillationelliptic}}
We adapt the classical strategy of Moser~\cite{Moser60} to the non-uniformly elliptic and discrete setting. As in the parabolic case, a key ingredient is local boundedness for non-negative \BR2 subharmonic functions.\ER 
\begin{theorem}\label{T:ellipticboundedness}
Fix $d\geq2$, $\omega\in\Omega$ and let $p,q\in(1,\infty]$ be such that $\frac1p+\frac1q<\frac2{d-1}$. 
\begin{enumerate}[(a)]
 \item Fix $d\geq3$.  For every $\gamma\in(0,1]$, there exists $c=c(d,p,q,\gamma)\in[1,\infty)$ such that for every $u>0$ satisfying $-\mathcal L^\omega  u\leq0$ in $B(2n)$, $n\in \mathbb N$ 
 %
\begin{equation}\label{est:T:boundl1}
 \max_{x\in B(n)}u(x)\leq c\Lambda_{p,q}^\omega( B(2n))^{\frac{\delta+1}{2\delta\gamma}}\|u\|_{\underline L^{2p'\gamma}(B(2n))},
\end{equation}
where $\delta:=\frac{1}{d-1}-\frac1{2p}-\frac1{2q}>0$, $p':=\frac{p}{p-1}$ and for every bounded set $ S\subset\mathbb Z^d$ 
\begin{equation}\label{def:lambda}
\Lambda_{p,q}^\omega(S):=\|\omega\|_{\underline L^p(S)}\|\omega^{-1}\|_{\underline L^q(S)}.
\end{equation}
\item Fix $d=2$. There exists $c\in[1,\infty)$ such that for every $u>0$ satisfying $-\mathcal L^\omega u\BR2\leq0$ in $B(2n)$, $n\in \mathbb N$ 
\begin{equation}\label{est:T:boundl12d}
 \max_{x\in B(n)}u(x)\leq c\left(n\|\omega^{-1}\|_{\underline L^1(B(2n))}^\frac12\|\sqrt \omega \nabla u\|_{\underline L^2(B(2n))}+\|u\|_{\underline L^{1}(B(2n))}\right).
\end{equation}

\end{enumerate}

\end{theorem} 
\begin{proof}
In \cite[Corollary 1, Proposition 4]{BS19b} the corresponding estimates with $\max u$ replaced by $\max |u|$ are proven for harmonic functions (i.e.\ $u$ satisfying $\mathcal L^\omega u=0$) without any sign condition on $u$. The  proofs apply almost verbatim to non-negative \BR2 subharmonic functions \ER and thus are omitted here.
\end{proof}
\begin{theorem}[Weak elliptic Harnack inequality]\label{T:weakHarnack:elliptic}
Fix $d\geq2$, $\omega\in\Omega$, and $p,q\in(1,\infty]$ satisfying $\frac1p+\frac1q<\frac2{d-1}$. Let $u\geq0$ be such that $\BR2-\mathcal L^\omega u\geq0$ in $B(4n)$. Suppose that there exist $\e>0$ and $\lambda\in(0,1)$ such that
\begin{equation}\label{assumption:weakharnack:elliptic}
 |\{x\in Q(2n),\, u(x)\geq\e\})\geq \lambda |B(2n)|.
\end{equation}
There exists a constant $c=c(d,p,q)\in[1,\infty)$ such that
\begin{equation}\label{claim:osc:elliptic}
 u(x)\geq\gamma\quad\textrm{for all }x\in B(n),
\end{equation}
where
$$
\gamma= \e\begin{cases}\exp\bigl(-c\lambda^{-1}\Lambda_{1,1}^\omega(B(4(n)))^{\frac12}\bigr)&\textrm{if }d=2,\\
\exp\bigl(-c\lambda^{-1}\Lambda_{p,q}^\omega(B(4(n)))^{\frac12+\frac{\delta+1}{\delta}p'(\frac12+\frac1q-\frac1d)}\bigr)&\textrm{if } d\geq3.\end{cases}
$$
\end{theorem}


\begin{proof}[Proof of Theorem~\ref{T:weakHarnack:elliptic}]

Without loss of generality, we assume $\e=1$. Throughout the proof we write $\lesssim$ if $\leq$ holds up to a positive constant depending on $d$, $p$, and $q$.
Consider the function $x\mapsto W(x):=g(u(x))$, with $g$ being defined in~\eqref{def:g}. 

\step 1 $W$ is \BR2 subharmonic\ER, i.e.\ $-\mathcal L^\omega \BR2 W\leq0$ in $B(4n)$.

This follows from Step~1 of the proof of Theorem~\ref{T:weakHarnack}.

\step 2 Fix $d\geq3$. We claim that there exists $c=c(d,q)\in[1,\infty)$ such that
\begin{align}\label{est:weakHarnackpoincare:elliptic}
\|W\|_{\underline L^Q(B(2n))}^2\leq c\lambda^{-2}\|\omega^{-1}\|_{\underline L^{q}(B(4n))}\|\omega\|_{\underline L^1(B(4n))},
\end{align}
where $\frac1Q=\frac12+\frac1{2q}-\frac1d$. Indeed, as in Proof of Lemma~\ref{L:weakHarnack}, Step~1, we obtain for any $\eta:\mathbb Z^d\to[0,1]$ with $\eta=0$ in $\mathbb Z^d\setminus B(4n-1)$
\begin{equation}\label{est:energywH1:elliptic}
 |B(2n)|^{-1}\sum_{\ee\in\mathbb B^d}\omega(\ee)\varphi_{\eta}(\ee)(\nabla W(\ee))^2\leq c\|\omega\|_{\underline L^1(B(4n))}\|\nabla \eta\|_{L^\infty}^2({\rm osr}\,\eta)^2.
\end{equation}
Choosing a suitable cut-off function satisfying $\eta=1$ in $B(2n)$, $\|\nabla \eta\|_{L^\infty}\lesssim n^{-1}$ and $({\rm osr}(\eta))\leq2$, we deduce from \eqref{est:energywH1} that
\begin{align}\label{est:weakHarnacknablaw:elliptic}
\|\sqrt{\omega} \nabla W\|_{\underline L^2(B(2n))}^2\leq& cn^{-2}\|\omega\|_{\underline L^1(B(4n))}.
\end{align}
Moreover, appealing to Sobolev inequality as in \eqref{ineq:weightedsobharnack}, we find $c=c(d,q)\in[1,\infty)$ satisfying
\begin{align*}
\|W\|_{\underline L^{Q}(B(2n))}^2\leq cn^2\lambda^{-2}\|\nabla W\|_{\underline L^{\frac{2q}{q+1}}(B(2n))}^2\leq cn^2\lambda^{-2}\|\omega^{-1}\|_{\underline L^q (B(2n))}\|\sqrt\omega \nabla W\|_{\underline L^2(B(2n))}^2,
\end{align*}
and claim \eqref{est:weakHarnackpoincare:elliptic} follows by \eqref{est:weakHarnacknablaw:elliptic}

\step 3 Conclusion for $d\geq3$.

Combining \eqref{est:weakHarnackpoincare:elliptic} with estimate \eqref{est:T:boundl1} (with $\gamma=\frac{Q}{2p'}$) and the fact that $W$ is subharmonic, we obtain
\begin{equation}\label{est:ellipticosc:dgeq3}
 \max_{x\in B(n)}W(x)\lesssim \lambda^{-1}\Lambda_{p,q}^\omega( B(2n))^{p'\frac{\delta+1}{2\delta}(1+\frac1q-\frac2d)}\|W\|_{\underline L^Q(B(2n))}\lesssim\lambda^{-1}\Lambda_{p,q}^\omega( B(4n))^{p'\frac{\delta+1}{2\delta}(1+\frac1q-\frac2d)+\frac12}.
\end{equation}
Estimate \eqref{claim:osc:elliptic} (for $d\geq3$) follows from \eqref{est:ellipticosc:dgeq3} and the definition of $W$ and $g$.

\step 4 The case $d=2$.

Assumption \eqref{assumption:weakharnack:elliptic} and a suitable version of Sobolev inequality (see \eqref{ineq:weightedsobharnack}) yield
\begin{equation}
 \|W\|_{\underline L^1(B(2n))}\lesssim n\lambda^{-1}\|\nabla W\|_{\underline L^1(B(2n))}\leq n\lambda^{-1}\|\omega^{-1}\|_{\underline L^1 (B(2n))}^{\frac12}\|\sqrt\omega \nabla W\|_{\underline L^2(B(2n))}
\end{equation}
and in combination with \eqref{est:weakHarnacknablaw:elliptic} and \eqref{est:T:boundl12d} (using that $W$ is subharmonic by Step~1)
\begin{align*}
\|W\|_{L^{\infty}(B(2n))}\lesssim \lambda^{-1}\|\omega^{-1}\|_{\underline L^1 (B(2n))}^{\frac12}\|\omega\|_{\underline L^1 (B(2n))}^{\frac12}.
\end{align*}
The claimed estimate follows by the definition of $W$.
\end{proof}

\begin{proof}[Proof of Theorem~\ref{P:oscillationelliptic}]
Appealing to the weak Harnack inequality Theorem~\ref{T:weakHarnack:elliptic} the proof follows by the same argument as in the parabolic case, see Theorem~\ref{T:oscillation}. 
\end{proof}

\section*{Acknowledgments}

The authors were supported by the German Science Foundation DFG in context of the Emmy Noether Junior Research Group BE 5922/1-1. M.S. thanks Martin Slowik for useful discussion on the subject and for providing an early version of \cite{CS19}.

\appendix

\section{Technical estimates}

We recall some estimates, \BR2 mainly proven in \cite[Lemma A.1]{ADS15}, \ER that we used in the proof of Theorem~\ref{T:bound}.
\begin{lemma} 
For $a\in\R$ and $\alpha\in \R\setminus\{0\}$, set $\tilde a_\alpha=|a|^\alpha {\rm sign}\, a$.
\begin{enumerate}[(i)]
\item For all $a,b\in\R$ and any $\alpha,\beta\neq0$
\begin{equation}\label{est:d:chain:1}
 |\tilde a_\alpha-\tilde b_\alpha|\leq \left(1\vee \left|\frac\alpha\beta\right|\right)|\tilde a_\beta-\tilde b_\beta|(|a|^{\alpha-\beta}+|b|^{\alpha-\beta})
\end{equation}
\item For all $a,b\geq0$ and $\alpha>\frac12$
\begin{equation}\label{est:d:chain:2}
 (a^\alpha- b^\alpha)^2\leq \biggl|\frac{\alpha^2}{2\alpha-1}\biggr|(a-b)(a^{2\alpha-1}- b^{2\alpha-1})
\end{equation}
\item For all $a,b\geq0$ and $ \alpha\ge 1$
\begin{equation}\label{est:d:chain:3}
 (a^{2\alpha-1}+b^{2\alpha-1})|a-b|\leq  |a^\alpha- b^\alpha|(a^\alpha+b^\alpha).
\end{equation}
\end{enumerate}

\end{lemma}

\begin{proof}
Parts (i) and (ii) are contained in \cite[Lemma A.1]{ADS15}. We provide the argument for~\eqref{est:d:chain:3}: We start with few reductions. First, by symmetry we can assume $a \ge b$ and the case $b=0$ being trivial allows us to farther assume $b>0$. By homogeneity of both sides of the inequality, we can farther assume $b = 1$, in which case the inequality reads
 \begin{equation*}
  a^{2\alpha} - a^{2\alpha-1} + a - 1 = (a^{2\alpha-1} + 1)(a-1) \le (a^\alpha-1)(a^\alpha+1) = a^{2\alpha} - 1,
 \end{equation*}
 which reduces to $a(1-a^{2(\alpha-1)}) = -a^{2\alpha-1} + a \le 0$, which using $2(\alpha-1) \ge 0$ is equivalent to $a \ge 1$, thus completing the argument. 
\end{proof}

Finally we recall a technical estimate given in \cite{CS19} that we used in the proof of Lemma~\ref{L:weakHarnack}.
\begin{lemma}[{\cite[Lemma A.1]{CS19}}]\label{L:appendixMartin}
Let $g\in C^1(0,\infty)$ be a convex, non-increasing function. Assume the $g'$ is piecewise differentiable and that there exists $\gamma\in(0,1]$ such that $\gamma g'(r)^2\leq g''(r)$ for a.e.\ $r\in(0,\infty)$. Then, for all $x,y>0$ and $a,b\geq0$
\begin{align*}
&-(b^2g'(y)-a^2g'(x))(y-x)\\
\leq& \begin{cases}
-\frac\gamma2(\min\{a^2,b^2\}(g(y)-g(x))^2+\frac2\gamma\max\{\frac{a^2}{b^2},\frac{b^2}{a^2}\}(b-a)^2&\mbox{if $\min\{a,b\}>0$,}\\
\max\{-xg'(x),-yg'(y)\}(b-a)^2&\mbox{if $\min\{a,b\}=0$}
\end{cases}
\end{align*}

\end{lemma}

\end{document}